\documentclass[11pt]{article}

\usepackage[utf8]{inputenc}
\usepackage{fullpage,graphicx,amssymb,amsmath,psfrag,xcolor,enumerate, float, amsthm,enumitem,mathtools,setspace,authblk}

\usepackage{natbib}

\allowdisplaybreaks

\usepackage[colorlinks,citecolor=bbluegray,linkcolor=ddarkbrown,urlcolor=blue,breaklinks]{hyperref}
\usepackage{bbding}
\usepackage{subfigure}

\usepackage{algorithmic,algorithm}

\mathtoolsset{showonlyrefs}

\usepackage[margin=2.5cm]{geometry}

\makeatletter
\g@addto@macro\normalsize{%
  \setlength\abovedisplayskip{7pt}
  \setlength\belowdisplayskip{7pt}
  \setlength\abovedisplayshortskip{5pt}
  \setlength\belowdisplayshortskip{5pt}
}
\makeatother

\definecolor{ddarkbrown}{rgb}{0.5,0.2,0.05} \definecolor{bbluegray}{rgb}{0.05,0,0.5}

\newtheorem{theorem}{Theorem}[section]
\newtheorem{proposition}[theorem]{Proposition}

\newtheorem{lemma}[theorem]{Lemma}
\newtheorem{corollary}[theorem]{Corollary}
\newtheorem{assumption}[theorem]{Assumption}

\newcommand{\xbar}{\overline{x}}
\newcommand{\C}{\mathcal{C}}
\newcommand{\K}{\mathcal{K}}

\newcommand{\dlog}{ D_{L} }

\let\la\langle
\let\ra\rangle
\newcommand{\BSP}{\begin{equation*}\begin{split}}
\newcommand{\ESP}{\end{split}\end{equation*}}

\newcommand{\bigO}{\mathcal{O}}

\newcommand{\BEAS}{\begin{eqnarray*}}
\newcommand{\EEAS}{\end{eqnarray*}}
\newcommand{\BEA}{\begin{eqnarray}}
\newcommand{\EEA}{\end{eqnarray}}
\newcommand{\BEQ}{\begin{equation}}
\newcommand{\EEQ}{\end{equation}}

\newcommand{\BEQS}{\begin{equation*}}
\newcommand{\EEQS}{\end{equation*}}

\newcommand{\BIT}{\begin{itemize}}
\newcommand{\EIT}{\end{itemize}}
\newcommand{\BNUM}{\begin{enumerate}}
\newcommand{\ENUM}{\end{enumerate}}

\newcommand{\BA}{\begin{array}}
\newcommand{\EA}{\end{array}}

\newcommand{\reals}{{\mathbb R}}

\newcommand{\symm}{\mathbb{S}}  

\newcommand{\Tr}{ {\rm Tr} }

\newcommand{\argmin}{\mathop{\rm argmin}}

\title{\textbf{Convex quartic problems: homogenized gradient method and preconditioning}}

\author[1]{Radu-Alexandru Dragomir}
\author[2]{Yurii Nesterov}

\affil[1]{EPFL, Switzerland}
\affil[2]{CORE, UCLouvain, Belgium}

\numberwithin{equation}{section}

\begin{document}

\maketitle
\begin{abstract}
We consider a convex minimization problem for which the objective is the sum of a homogeneous polynomial of degree four and a linear term. Such task arises as a subproblem in algorithms for quadratic inverse problems with a difference-of-convex structure. We design a first-order method called \textit{Homogenized Gradient}, along with an accelerated version, which enjoy  fast convergence rates of respectively $\bigO(\kappa^2/K^2)$ and $\bigO(\kappa^2/K^4)$ in relative accuracy, where $K$ is the iteration counter. The constant $\kappa$ is the {quartic condition number} of the problem.

Then, we show that for a certain class of problems, it is possible to compute a preconditioner for which this condition number is $\sqrt{n}$, where $n$ is the problem dimension. To establish this, we study the more general problem of finding the best quadratic approximation of an $\ell_p$ norm composed with a quadratic map. Our construction involves a generalization of the so-called Lewis weights.
\end{abstract}

\section{Introduction}

\subsection{Problem setup}

We consider the convex minimization problem
\BEQ\label{eq:p}\tag{P}
    \min_{x \in \K} f(x) \triangleq \rho(x) - \la c,x \ra
\EEQ
where $\K$ is a closed convex cone of a Euclidean space $E$, $c$ is a nonzero vector of $E$, and $\rho$ is a convex quartic form, meaning that $\rho$ is convex and
\[ \rho(x) = Q[x,x,x,x], \]
for some symmetric 4-linear form $Q:E^4\rightarrow \reals$. We assume that $\rho$ satisfies the following bounds with respect to a given Euclidean norm $\|\cdot\|$:
\BEQ\label{eq:q_cond}
        \alpha^2 \|x\|^4 \leq \rho(x) \leq \beta^2 \|x\|^4,\quad \forall x \in E,
\EEQ
where $0<\alpha\leq \beta$.
We call $\kappa \triangleq \frac{\beta}{\alpha}$ the \emph{quartic condition number} of $\rho$. Problem \eqref{eq:p} arises as a subproblem in some algorithms for solving a large class of quartic optimization problems, as we detail in Section \ref{s:setup}. Applications include quadratic least squares, phase retrieval, distance matrix completion and matrix factorization.

 We propose to study first-order methods for solving \eqref{eq:p}. In standard gradient methods, the objective gap is bounded as $\bigO(LR^2/K)$, where $K$ is the iteration counter, $R$ is the initial distance to the optimal point and $L$ is the Lipschitz constant of the gradient; accelerated methods allow the convergence speed to be improved to $\bigO(LR^2/K^2)$ \citep{Nesterov1983}. However, these methods require the gradient of $f$ to be globally Lipschitz continuous, which is not the case for quartic function $\rho$. This can be remedied by constraining the optimization on a bounded set, but this generally leads to a highly over-estimated Lipschitz constant.

In this work, we adopt a different approach by exploiting the polynomial quartic structure of convex function $\rho$ and the regularity condition \eqref{eq:q_cond}.

\subsection{Contributions and outline}
Our main goal is to design first-order methods which enjoy fast convergence rates for solving Problem \eqref{eq:p}. We list below our main contributions.

\begin{itemize}
    \item \textbf{Regularity properties of convex quartic functions} (Section \ref{s:quartic_prop}): using the polynomial structure of $\rho$ and the condition \eqref{eq:q_cond}, we establish two crucial properties for our analysis. First, we show that $\rho$ satisfies a \textit{uniform convexity property} of degree 4. Then, we prove that the function $\sqrt{\rho}$ is convex and has Lipschitz continous gradient.
    \item \textbf{Homogenized gradient descent} (Section \ref{s:hom}): as a next step, we show that a solution to Problem \eqref{eq:p} can be obtained by solving a \textit{homogenized problem}
    \BEQ \label{eq:p_hom_intro} \tag{H}
    \min_{ \substack{y \in \K \\ \la c, y \ra = 1} } \sqrt{\rho(y)}.
    \EEQ
    As $\sqrt{\rho}$ is convex and has Lipschitz gradient, we can apply projected gradient descent to \eqref{eq:p_hom_intro}. The uniform convexity of $\rho$ allows to establish faster convergence rates than in the standard smooth convex case. We prove that homogenized gradient descent enjoys a $\bigO(\kappa^2/K^2)$ rate, and that an accelerated technique allows to improve this speed to $\bigO(\kappa^2/K^4)$. The key to obtain this fast sublinear rate is the \textit{restart} mechanism, which has been studied before in situations similar to uniform convexity \citep{Nemirovski1983,Roulet2017,Nesterov2020}.

    \item \textbf{Optimal preconditioning} (Section \ref{s:precond}): as the efficiency of our methods depend on the condition number $\kappa$, we look to improve it by choosing an appropriate Euclidean norm of the form $\|\cdot\|_B$ induced by a positive definite operator $B$.
    We study the case where $\rho$ is of the form
    \BEQ\label{eq:rho_qip_intro}
        \rho(x) = \sum_{i=1}^m \la x, B_i x \ra^2.
    \EEQ
  This structure arises in our main application of interest, the \textit{quadratic inverse problems}. We show that there exists a Euclidean norm $\|\cdot\|_{B^*}$ induced by an operator $B^*$ for which  $\kappa = \sqrt{n}$. Our analysis can be applied to the problem of approximating the function $x \mapsto\left(\sum_{i=1}^m\la B_i x, x\ra^{p}\right)^{1/p}$, $p \geq 1$, by a Euclidean norm. We show that there exists a norm with condition number $n^{1-1/p}$, which is not improvable in general. This norm is induced by an operator $B^* = \sum_{i=1}^m \tau_i^* B_i$, where $\tau^*$ satisfies a given fixed-point equation.




     The coefficients $\tau_i$ can be seen as a generalization of the called \textit{Lewis weights} used in $\ell_p$ subspace approximation \cite{Lewis1978FiniteDS}.

    We then show that $\tau^*$ can be computed with a fixed-point algorithm, which runs in $\mathcal{\tilde{O}}(mrn^2)$ time (the notation $\mathcal{\tilde{O}}$ hiding logarithmic factors), where $n = \dim E$ and $r$ is the maximal rank of the $B_i$'s. Using $B^*$ allows significant gains over the simpler choice $\overline{B}=\sum_{i=1}^m B_i$ if $m \gg n$ and the set of matrices $B_1,\dots B_m$ has \textit{high coherence}.

    \item \textbf{Numerical experiments} (Section \ref{s:numerical}): finally, we apply our algorithms to several problems of the form \eqref{eq:rho_qip_intro}. We generate instances of $\rho$ with various condition number and coherence values in order to demonstrate the impact of homogenization, acceleration and preconditioning. 
\end{itemize}

\subsection{Related work}

\subsubsection{Polynomial growth and fast convergence rates} The use of a lower approximation of the objective function in order to improve convergence rates has been long studied in first-order optimization. When the lower bound is quadratic, this amounts to strong convexity and allows to obtain linear convergence. If it is polynomial of some other degree, linear convergence is out of reach but {faster sublinear rates} can be obtained. Several variants of polynomial lower bound have been studied, such as \textit{uniform convexity} \citep{Juditsky2014,Nesterov2020}, Kurdyka–Łojasiewicz property \citep{Bolte2014PALM,Bolte2017errour_bounds,Roulet2017} or other conditions \citep{Freund2018}; see also \citep{Doikov2022} for lower complexity bounds in this setting. Compared to previous work, the novelty of our setting is that the uniform convexity property does not hold for the objective function that we are trying to minimize, but its square (see Section \ref{ss:convrates}). 
Minimization of convex quartic polynomials has also been tackled by higher-order tensor methods \citep{Bullins2018,Nesterov2020}.

\subsubsection{Bregman methods} Another approach which has been applied to quartic problems is relatively-smooth minimization \citep{Bauschke2017,Lu2016}. This method relies on the choice of a reference function $h$ whose Hessian upper bounds that of $f$, allowing to obtain possibly a better efficiency than the standard Euclidean gradient method. In the case of quartic functions, this technique has been applied by choosing $h$ to be a simple polynomial of degree 4 \citep{Bolte2018,Mukkamala2019b,Dragomir2019}. For convex functions, such scheme enjoys a convergence rate of $\bigO(LD/K)$, where $L$ is the relative Lipschitz constant and $D$ the initial Bregman distance to the optimum. 

A main drawback of Bregman methods is that they are difficult to accelerate: the lower bound by \citep{Dragomir2019a} guarantees that no Bregman-type algorithm can do better than $\bigO(LD/K)$ for generic reference functions $h$. Even with additional regularity assumptions, known acceleration results are limited to local or asymptotic regimes \citep{Hanzely2018,Hendrikx2020}. Another caveat is that relatively-smooth minimization does not allow to use the lower bound on $\rho$ in \eqref{eq:q_cond} in order to improve convergence guarantees\footnote{One could wonder if the term $\alpha\|x\|^4$ implies a \textit{relative strong convexity} \citep{Lu2016} with respect to a well-chosen quartic function $h$. In turns out that this is not true in general, apart from trivial choices such as $h =\rho$. This motivates the need to consider the weaker notion of \textit{uniform convexity}, which will not be enough to guarantee linear convergence but will lead to improved sublinear rates.}.

\subsubsection{Preconditioning and optimal quadratic norms} The idea of finding the best quadratic norm to approximate a convex set can be traced back to John \citep{John1948}. This amounts to computing the \textit{minimum-volume ellipsoid} enclosing a given set. Efficient procedures for computing such ellipsoid have been proposed in \citep{Khachiyan1996,Anstreicher1999}; see also \citep{Todd2016} for a recent survey. This principle has been applied to linear programming \citep{Nesterov2008rounding}. In our case, the direct applications of such methods would yield a quartic condition number of $n$, where $n$ is the dimension of $E$. In our work, we propose a variant specialized to quartic functions, allowing to compute a quadratic norm for which the condition number of $\rho$ is reduced to $\sqrt{n}$.

\subsubsection{Lewis weights} In the case where the matrices $B_i$ are of rank 1, i.e. $B_i = a_i a_i^T$ for some $a_i \in \reals^n$, then the weights $\tau^*$ satisfying \eqref{eq:FP_intro} are called the \textit{Lewis weights} of the matrix $A= [a_1^T \dots a_m^T]^T \in\reals^{m\times n}$. They were initially introduced in \citep{Lewis1978FiniteDS} and studied further in \citep{Bourgain1989LW,Talagrand1995,Cohen2015lewis}. 
These weights are used as sampling probabilities to approximate the $\ell_p$-norm $\|Ax\|_p$ by selecting a random subset of the rows of $A$. In our work, we show that the same quantities can be used to build an approximation of functions of the type $\|Ax\|_{2p}$ by a Euclidean norm $\|x\|_B = \la Bx, x\ra^{\frac{1}{2}}$.
The algorithm that we use for computing $\tau^*$ is inspired by a fixed point scheme in \citep{Cohen2015lewis}.

\subsection{Notation and definitions} Throughout the paper, $\K$ denotes a convex cone of a Euclidean space $E$. We assume that the norm $\|\cdot\|$ is induced by a self-adjoint positive definite operator $B:E\rightarrow E^*$ as
\[ \|x\| = \|x\|_B = \la Bx, x\ra, \]
for any $x \in E$, where $\la \cdot, \cdot \ra$ is the canonical inner product. We will write $\|\cdot\|_B$ when we want to insist on a specific choice of matrix $B$, and $\|\cdot\|$ when it is clear from context.

A function $Q:E^4\rightarrow \reals$ is a 4-symmetric linear map if it is linear in each variable and invariant with respect to permutation of the variables. We use the short notation $Q[x]^4 = Q[x,x,x,x]$ and $Q[x]^2[y]^2 = Q[x,x,y,y]$. 

The notation $[t]_+ = \max(t,0)$ denotes the positive part of a real number~$t$.
We say that a differentiable function $g:E\rightarrow \reals$ is \textit{uniformly convex} of degree $p\geq 2$ with parameter $\sigma$ if for any $x,y \in E$ we have
\[ g(x) - g(y) - \la \nabla g(y), x-y \ra \geq \frac{\sigma}{p}\|x-y\|^p.\]

\section{Motivation and examples}\label{s:setup}

In this section, we give examples of relevant applications involving minimization of problems of the form of \eqref{eq:p}.

\subsection{DC algorithm for quadratic inverse problems}

The auxiliary problems of the form \eqref{eq:p} arise in the algorithms for \textit{quadratic inverse problems} 
\BEQ\label{eq:noncvx_p}\tag{QIP}
    \min_{x \in \K}\, F(x) \triangleq  \sum_{i=1}^m \left( \la  x, B_i x\ra - d_i \right)^2
\EEQ
where $B_1\dots B_m$ are symmetric linear operators on $E$ and $d_1\dots d_m \in \reals$ some target values. Problem \eqref{eq:noncvx_p} is generally nonconvex and occurs in various tasks of data science and signal processing; we detail some examples below.

Expanding the square, $F$ rewrites as a difference of two functions $\rho,\phi$:
\BEQ\label{eq:dc}
F(x) = \underbrace{\sum_{i=1}^m \la x, B_i x\ra ^2}_{\rho(x)} - \underbrace{\sum_{i=1}^m \left(2 d_i \la x, B_i x\ra - d_i^2 \right)}_{ \phi(x) }.
\EEQ
We make the following assumption:
\begin{center}
\mbox{$\rho$ and $\phi$ are \textbf{convex functions}.} 
\par
\end{center}
This holds for instance if $B_1,\dots B_m$ are positive semidefinite operators and $d_i \geq 0$ for $ i = 1\dots m$.
We say that $F$ has a \textbf{difference-of-convex} (DC) structure. Convexity of $\phi$ allows to design a convex global majorant of $F$ around any point $\xbar$ with the inequality
    \BEQ
            F(x) \leq  \rho(x) - \phi(\xbar) - \la \nabla \phi(\xbar),x-\xbar  \ra.
    \EEQ
    Then the DC algorithm \citep{Tao1986,An2005} simply consists in successive minimization of this global majorant: choose $x_0 \in \K$ and compute for $k = 0,1,\dots$
    \BEQ\label{eq:DC}\tag{DC}
        x_{k+1} \in \argmin_{x \in \K} \rho(x) - \la \nabla \phi(x_k),x \ra.
    \EEQ
    We recognize here a problem of the form \eqref{eq:p}, with $c = \nabla \phi(x_k)$ and $\rho$ is the convex quartic polynomial
    \BEQ\label{eq:def_rho} \rho(x)= \sum_{i=1}^m \la x, B_i x\ra ^2. 
    \EEQ
    The corresponding symmetric 4-linear form is $Q[u,v,w,z] = \sum_{i=1}^m \la u,B_iv \ra \la w,B_iz \ra$.

    The DC algorithm has mostly been studied in the context of polyhedral functions. It is quite general as any function with Lipschitz continuous gradient can be proven to have a DC decomposition. Its efficiency depends on how well we can solve the subproblems defining $x_{k+1}$ in \eqref{eq:DC}. In this work, we show that the case where $\rho$ is a convex quartic polynomial is a favorable situation. 

    We show that, when the operators $B_i$ are positive semidefinite, function $\rho$ satisfies the quartic conditioning assumption \eqref{eq:q_cond}. In particular, the constant $\alpha$ is positive when the operator $\sum_{i=1}^m B_i$ is positive definite. 

        \begin{proposition}\label{prop:cond_number}
          Assume that the operators $B_1,\dots B_m$ are all positive semidefinite and denote $\bar B = \sum_{i=1}^m B_i$. Then the function $\rho$ defined in \eqref{eq:def_rho} satisfies
          \[
          \frac{\lambda_{\rm min}(\bar B)^2}{m} \|x\|_2^4 \leq \rho(x) \leq \lambda_{\rm max}(\bar B)^2 \|x\|_2^4,
          \]
          for $x\in E$, where $\lambda_{\rm min}$ and $\lambda_{\rm max}$ denote the smallest and largest eigenvalues.
        \end{proposition}

        \begin{proof}
          The lower bound follows from Cauchy-Schwarz:
          \[
            \left(\sum_{i=1}^m \la x,B_ix \ra \right)^2 \leq m \sum_{i=1}^m \la x,B_ix \ra^2.
          \]
          For the upper bound, since $\la B_ix,x \ra \geq 0$ for each $i$ we have
          \[
            \sum_{i=1}^m \la x,B_ix \ra^2 \leq \left(\sum_{i=1}^m \la x,B_ix \ra \right)^2.
          \]
          We conclude by using $\lambda_{\rm min}(\bar B) \|x\|_2^2 \leq \sum_{i=1}^m \la x,B_ix \ra \leq \lambda_{\rm max}(\bar B) \|x\|_2^2$.

        \end{proof}

        Therefore, the quartic condition number of $\rho$ is upper bounded by $\sqrt{m} \kappa_2(\bar B)$, where $\kappa_2(\bar B)$ is the condition number in the usual sense of operator $\bar B$. In Section \ref{s:precond}, we refine this estimate using the \textit{coherence} of the set $(B_1,\dots B_m)$.

        We now list a few examples.

        \paragraph{Example 1: phase retrieval} We wish to recover a complex signal $x^* \in \mathbb{C}^N$ from phaseless quadratic measurements $d_i = |\la q_i, x^*\ra|^2$ where $q_i \in \mathbb{C}^N$ for $i = 1 \dots m$. This is a fundamental signal processing problem with applications in optimal imaging, cristallography and astronomy \citep{Shechtman2014phase,Candes2015,Sun2017}. 
        A popular method for performing phase retrieval is to solve the nonconvex optimization problem
        \BEQ\label{eq:phase_ret}
            \min_{x \in \mathbb{C}^N}\, \frac{1}{m} \sum_{i=1}^m (|\la q_i, x\ra|^2  - d_i)^2.
        \EEQ
        The DC structure arises as $\phi(x) =  \frac{1}{m} \sum_{i=1}^m (2 d_i|\la q_i, x\ra|^2-d_i^2)$ is a convex quadratic and the quartic part is
        \begin{align*}
           \rho(x) &= \frac{1}{m} \sum_{i=1}^m |\la q_i,x \ra|^4.
        \end{align*}
        The operator $B_i$ is $q_i q_i^H$ (where $q^H$ denotes Hermitian transpose). Since the vectors $q_1, \dots q_m$ are typically obtained from overcomplete Fourier bases or Gaussian distributions, the sum $\sum_{i=1}^m B_i$ is generally invertible and well-conditioned \citep{Candes2015}.


\newcommand{\centmat}{\mathcal{C}}

\paragraph{Example 2: distance matrix completion (DMC)} In this problem, we wish to recover the position of $N$ points $x_1^*,\dots,x_N^* \in \reals^r$ from the partial knowledge of the pairwise distances $d_{ij} = \|x_i^*-x_j^*\|_2^2$ for $i,j \in \Omega$, where $\Omega \subset \{1\dots N\}\times \{1\dots N\}$. EDMC has applications in sensor network localization and biology; see \cite{Fang2012,Qi2013,Dokmanic2015} and references therein. There are several ways to solve this task, including semidefinite programming; however, for large-scale problems, the following nonconvex formulation is popular due to its smaller dimension:
\BEQ\label{eq:DMC}\tag{DMC}
\min_{
\substack{
  X\in\reals^{N\times r}\\
  X^T \mathbf{1} = 0
}
} \sum_{(i,j) \in \Omega} \left( \|x_i-x_j\|_2^2 -  d_{ij}\right)^2,
\EEQ
where $x_1,\dots x_N$ are the rows of matrix $X$, and the centering constraint allows to remove the translation invariance. The convex quartic function $\rho$ can be written in this case
\[\rho(X) = \sum_{(i,j)\in \Omega} \|x_i-x_j\|_2^4 + \frac{1}{N^2} \big\|\sum_{i=1}^N x_i\big\|^4.
\] 
The addition of the second term, which is 0 on the feasible set of \eqref{eq:DMC}, allows to make $\rho$ positive definite. 
The linear operator $\bar B$ as defined in Proposition \ref{prop:cond_number} satisfies 
 \begin{align*}
  \la\bar B X,X\ra &= \sum_{(i,j \in \Omega)} \|x_i-x_j\|_2^2 +\frac{1}{N} \big\|\sum_{i=1}^N x_i\big\|^2\\
  &=\sum_{k=1}^r \la \left(\mathcal{L}(\Omega) + \frac{1}{N} \mathbf{1}_N\mathbf{1}_N^T\right) x_{\cdot k},x_{\cdot k} \ra,
\end{align*}
where $x_{\cdot k}$ is the $k$-th column of $X$ and $\mathcal{L}(\Omega) \in \reals^{N\times N}$ is the Laplacian matrix of the graph induced by $\Omega$ \cite{chung1997spectral}. 

The smallest eigenvalue of $\mathcal{L}(\Omega)$ is $0$ with eigenvector $\mathbf{1}_N$. The second smallest eigenvalue, which we denote $\lambda_1(\Omega)$, is often called the \textbf{graph spectral gap}. The largest eigenvalue is bounded by $4\bar d(\Omega)$, where $\bar d(\Omega)$ is the maximal degree of the graph. We deduce that the eigenvalues of operator $\bar B$ are located in $[\lambda_1(\Omega), 4\bar d(\Omega) + 1]$. It follows from Proposition \ref{prop:cond_number} that for every $X \in \reals^{N \times r}$,
\[
    \left(\frac{\lambda_1(\Omega)^2}{|\Omega|+1}\right) \|X\|_F^4 \leq \rho(X) \leq (4\bar d(\Omega)+1)^2 \|X\|_F^4.
\]
The spectral gap $\lambda_1(\Omega)$ is nonzero if the graph is connected \cite{chung1997spectral}, which is a necessary condition for Problem \eqref{eq:DMC} to be solvable.

\paragraph{Example 3: symmetric nonnegative matrix factorization (SymNMF).}
Let $M \in \reals^{n\times n}_+$ be a nonnegative matrix, which we assume to be {positive semidefinite}. We wish to find a symmetric low-rank decomposition $M \approx XX^T$ where the factor $X\in \reals^{n\times r}_+$ has nonnegative entries and $r$ is a small target rank. This task has relevant applications in graph clustering \citep{He2011,Kuang2015}. SymNMF is usually performed by solving the nonconvex problem
\BEQ\label{eq:symnmf}\tag{SymNMF} \min_{X \in \reals^{n \times r}_+} \|XX^T - M\|_{ F}^2 = \|XX^T\|_F^2 - \left( 2\la MX, X \ra - \|M\|_F^2\right)\EEQ
where $\|\cdot\|_F$ denotes the Frobenius norm. Then $\mathcal{K} = \reals^{n \times r}_+$ is the nonnegative orthant, the function ${\phi(X) = 2 \la MX, X\ra -\|M\|_F^2}$ is a convex quadratic (since $M$ is positive semidefinite) and the quartic part is the convex function $\rho(X) = \|XX^T\|_F^2$. In this situation, Proposition \ref{prop:cond_number} does not apply as the individual operators $B_i$ are not positive semidefinite. However, we can still estimate the quartic condition number. Using the properties of the Frobenius inner product we have
\[ \rho(X) = \text{Tr}(XX^TXX^T) = \text{Tr}(X^TXX^TX) = \|X^T X\|_F^2 = \sum_{i,j=1}^r (X_{\cdot,i}^T X_{\cdot,j})^2,\]
where $X_{.,i}$ is the $i$-th column of $X$. Hence $\rho$ satisfies the bounds
\begin{align*}
    \rho(X) &\leq \sum_{i,j=1}^r \|X_{\cdot,i}\|_F^2 \|X_{\cdot,j}\|_F^2 = \|X\|_F^4,\\
    \rho(X) &\geq \sum_{i=1}^r \|X_{\cdot,i}\|_F^4 \geq \frac{1}{r} \left(\sum_{i=1}^r \|X_{\cdot,i}\|_F^2\right)^2 = \frac{1}{r}\|X\|_F^4.
\end{align*}
The condition number is hence $\sqrt{r}$, which is small in most applications.

\subsection{Statistical Bregman preconditioning for stochastic and distributed optimization}

We consider the same setup as in the previous section, that is, a quadratic inverse problem of the form \eqref{eq:noncvx_p} with a difference-of-convex structure. We mention another technique which is especially suited for huge-scale and potentially distributed problems.

Assume that the measurement operators $B_1\dots B_m$ are positive semidefinite and generated as i.i.d. samples from a certain statistical distribution. Then, one can subsample this distribution by choosing a smaller index set $I \subset \{1\dots m\}$ and form an approximation $\tilde{\rho}$ to $\rho$ as
\[ \tilde{\rho}(x) = \frac{m}{|I|}\, \sum_{i \in I} \la x, B_i x\ra^2. \]
With the right statistical assumptions, one can prove that when $|I|$ is sufficiently high, \textit{relative smoothness and strong convexity} hold with high probability:
\BEQ\label{eq:spag} \mu_m \nabla^2\tilde{\rho}(x) \preceq \nabla^2 \rho(x) \preceq L_m  \nabla^2 \tilde{\rho}(x), 
\EEQ
where $\mu_m,L_m$ are the relative regularity constants. We refer the reader to \citep{shamir2014communication,Hendrikx2020} for detailed results on statistical preconditioning. Condition \eqref{eq:spag} allows to apply the Bregman (stochastic) gradient method with reference function $\tilde{\rho}$:
\BEQ\label{eq:breg_spag}
    x_{k+1} = \argmin_{x \in \K} \, \la g_k,x-x_k \ra + \frac{1}{\lambda_k} \left( \tilde{\rho}(x) - \tilde{\rho}(x_k) - \la \nabla \tilde{\rho}(x_k),x-x_k\ra \right),
\EEQ
where $g_k$ is a stochastic unbiased estimate of $\nabla F(x_k)$, and $\lambda_k >0$ is the step size. 

The choice of $I$ involves a tradeoff: if $I=\{1\dots m\}$ the problem \eqref{eq:breg_spag} can be as hard to solve as the original one \eqref{eq:noncvx_p}. However, choosing $I$ as a smaller subset can result in a significant performance gain compared to a standard Euclidean method. This approach can be advantageous in the context of distributed optimization, where the computation \eqref{eq:breg_spag} is performed by a centralized server and the goal is to reduce the overall number of stochastic gradients computed as they involve costly communication among machines \citep{Hendrikx2020}.

Problem \eqref{eq:breg_spag} is of the form \eqref{eq:p} with $c = \nabla \tilde{\rho}(x_k) - \lambda_k g_k$, which is another motivation for our work.

\section{Properties of convex quartic polynomials}\label{s:quartic_prop}

Before studying methods for solving Problem \eqref{eq:p}, we need to establish some key preliminary results on convex quartic polynomials. Recall that we assume that $\rho$ induced by the symmetric 4-linear form $Q:E^4\rightarrow \reals$ as $\rho(x)=Q[x]^4$, and that $\rho$ is convex.

\subsection{Convexity and smoothness of $\sqrt{\rho}$}    
We first recall the following fact about convex quartic forms.
    \begin{lemma}[Theorem 1 in \citep{Nesterov2022}]\label{lemma:quartic_cauchy} 
    Let $Q:E^4\rightarrow \reals$ be a symmetric 4-linear form for which the function $x\mapsto Q[x]^4$ is convex. Then for every $x,y \in E$ we have
        \begin{align} 
          0 &\leq Q[x]^2[y]^2,\\
        \label{eq:ineq_1}\left(Q[x]^2[y]^2\right)^2 &\leq Q[x]^4\, Q[y]^4, \\
        \label{eq:ineq_2} \left(Q[x]^3[y]\right)^2 &\leq Q[x]^4 Q[x]^2[y]^2.
        \end{align}
    \end{lemma}

This inequality implies that $\sqrt{\rho}$ has favorable properties for gradient methods.
    \begin{proposition}\label{prop:quartic_form_hessian}
        The function $g=\sqrt{\rho}$ is convex, differentiable and has Lipschitz continuous gradients with constant $6 {\beta}$.
    \end{proposition}
    \begin{proof}
        For $x \neq 0$, $\rho(x) > 0$ by \eqref{eq:q_cond} and therefore $g$ is twice differentiable at $x$. Its derivative along a direction $h\in E$ is
        \BEQ\label{eq:nablag}
          \la \nabla g(x),h \ra = \frac{2 Q[x]^3[h] }{(Q[x]^4)^{1/2}}.
        \EEQ
        Its second derivative writes for $h \in E$
        \[
             g''(x)[h,h] = \frac{6\,Q[x]^2[h]^2}{\left(Q[x]^4\right)^{1/2}} - \frac{4\left(Q[x]^3[h]\right)^{2}}{\left(Q[x]^4\right)^{3/2}}.
        \]
        Equation \eqref{eq:ineq_2} implies that
        $g''(x)[h,h] \geq \frac{2\,Q[x]^2[h]^2}{\left(Q[x]^4\right)^{1/2}} \geq 0, $ 
        which proves convexity on $E \setminus \{0\}$. Since $g$ is continuous at $0$, convexity on $E$ follows.
        
        Let us now prove Lipschitz continuity of gradients. Using Equation \eqref{eq:ineq_1} and the bound \eqref{eq:q_cond} yields that for $x \neq 0$,
        \begin{align*}
            g''(x)[h,h] & \leq \frac{6\,Q[x]^2[h]^2}{\left(Q[x]^4\right)^{1/2}} 
            \leq 6 \left(Q[h]^4\right)^{1/2} 
            \leq 6 {\beta} \|h\|^2.
        \end{align*}
   This proves that $\|\nabla g(x) - \nabla g(y)\| \leq 6 \beta \|x-y\|$ for every $x,y \in E$ such that $0$ does not belong to the segment $[x,y]$.
   
   Consider now the situation at $0$. Since $|g(0+h)| = |\sqrt{\rho(h)}|\leq \beta \|h\|^2$, $g$ is differentiable at 0 and $\nabla g(0) = 0$. Let us prove that $\nabla g$ is continuous at $0$. Using the expression \eqref{eq:nablag} and Lemma \ref{lemma:quartic_cauchy}, we have for $x \neq 0$ and $h \in E$
   \begin{align*}
      |\la \nabla g(x),h \ra| \leq 2
     \left(Q[x]^4 Q[h]^4\right)^{1/4} \leq 2\beta \|x\| \|h\|
   \end{align*}
   and therefore $\|\nabla g(x)\| \leq 2\beta \|x\|$, implying that $\nabla g$ is continuous at 0. Hence, by continuity and the previous point we have $\|\nabla g(x) - \nabla g(0)\| \leq 6\beta\|x\|$ for every $x \in E$. It remains to deal with the situation where $0 \in [x,y]$. In this case we have
   $$
\| \nabla g(x) - \nabla g(y) \| \leq \| \nabla g(x) \| + \| \nabla g(y) \| \leq 6\beta (\| x \| + \| y \|) = 6\beta \| x-y \|.
$$
\end{proof}

\subsection{Uniform convexity of $\rho$}

%
    
We show that $\rho$ satisfies the \textit{uniform convexity} property of degree 4.

    \begin{proposition}\label{prop:quartic_lower_bound}
        The function $\rho$ is uniformly convex of degree 4 with constant $4\alpha^2/3$: for every $x,y \in E$,
        \[ \rho(x) - \rho(y) - \la \rho'(y), x-y \ra  \geq \frac{\alpha^2}{3} \|x-y\|^4.
        \]
    \end{proposition}

    \begin{proof}
      Let $h = x-y$. By expanding $Q[y+h]^4$ we get
      \begin{align*}
        \rho(y+h) - \rho(y) - \la \rho'(y),h \ra &= Q[y+h]^4 - Q[y]^4 - 4 Q[y]^3[h]\\
        &=Q[h]^4+ 6Q[y]^2[h]^2 + 4Q[y][h]^3 \\
        &\geq \frac{\alpha^2}{3}\|h\|^4+\frac{2}{3}Q[h]^4+ 6Q[y]^2[h]^2 + 4Q[y][h]^3
      \end{align*}
      We show that the residual is nonnegative using the inequality $a^2 + b^2 \geq 2ab$ and Lemma \ref{lemma:quartic_cauchy}:
      \begin{align*}
      \frac{2}{3}Q[h]^4+ 6Q[y]^2[h]^2 + 4Q[y][h]^3  &\geq 2 \sqrt{4 Q[h]^4 Q[y]^2[h]^2} - |4Q[y][h]^3|\geq 0.
      \end{align*}
    \end{proof}

\section{Homogenized gradient descent}\label{s:hom}
We now design gradient methods for solving \eqref{eq:p}. Using the results from the previous sections, we first show that it can be transformed into a more favorable \textit{homogenized} problem, which satisfies the classical assumptions for smooth convex minimization.

From now on, we make an assumption ensuring that problem \eqref{eq:p} is not trivial.

\begin{assumption}\label{ass:polar}
  There exists $x \in \mathcal{K}$ such that $\la c,x \ra > 0$.
\end{assumption}

In other words, $c$ does not belong to the polar cone $\mathcal{K}^\circ$. If this was the case, the problem would minimized at 0.

\subsection{Homogenization}

\newcommand{\fhom}{f_{\rm hom}}
Recall that $f(x) = \rho(x) - \la c,x \ra$.
Let $\fhom$ the \emph{homogenization} of $f$ be defined for $y \in E$ as
\BEQ
    \fhom(y) = \min_{s \in \reals_+} f(s y)
\EEQ
 As $f(s y) = s^4 \rho(y) - s \la c, y\ra$, the one-dimensional minimization in $s$ can be computed easily. Writing $s(y)$ the solution to this minization, we have for $y\neq 0$
\BEQ\label{eq:def_tau}
\begin{split}
    &s(y) = \argmin_{s \geq 0} f(s y) = \left(\frac{[\la c, y\ra]_+}{4\rho\left(y\right)}\right)^{1/3},\\
    &\fhom(y) = f(s(y)y) = \frac{-3}{4^{4/3}} \left( \frac{[\la c,y \ra]_+^4}{\rho(y)} \right)^{1/3}.
\end{split}
\EEQ
Since Problem \eqref{eq:p} is defined on a cone $\K$, we have
\BEQ\label{eq:minimhom}
\min_{x \in \K}f(x) = \min_{y \in \K} \min_{s \geq 0} f(s y) = \min_{y \in \K} \fhom(y),
\EEQ
and the corresponding minimizers are
\BEQ\label{eq:minimizer_phom}
    \argmin_{x \in \K} f(x) = \Big\{ s(y) y \,:\, y \in \big\{\argmin_{y \in \K} \fhom(y) \big\}\Big\}.
\EEQ
Since $f_{\rm hom}(y) = 0$ if $\la c,y \ra \leq 0$, we can restrict the minimization to the set ${\{y \in \mathcal{K}\,:\, \la c,y \ra > 0\}}$, which is nonempty by Assumption \ref{ass:polar}. Furthermore, note that $f_{\rm hom}$ is a homogeneous function of degree 0. Since it is independent to rescaling, we can restrict to the set where $\la c,y \ra = 1$ and write
\BEQ
    \min_{y \in \K} \fhom(y) = \min_{y \in \K} \frac{-3}{4^{4/3}} \left( \frac{\la c,y \ra^4}{\rho(y)} \right)^{1/3}= \min_{\substack{y \in \K \\ \la c, y \ra = 1}}\, \frac{-3}{4^{4/3}\rho(y)^{1/3}} 
\EEQ
Note that the minimizer satisfies $\rho(y^*) > 0$.
Since any increasing transformation of the objective function preserves the minimizer, this problem has the same solution as
\BEQ\label{eq:p_hom}\tag{H}
    \min_{\substack{y \in \K\\ \la c, y \ra = 1}} \sqrt{\rho(y)}.
\EEQ
We summarize this discussion in the following result, which relates the quality of a solution to \eqref{eq:p_hom} to that of the original problem \eqref{eq:p} in \textit{relative accuracy}.

\begin{proposition}[Homogenization]\label{prop:hom}
    Let $y^*$ be a minimizer of \eqref{eq:p_hom}. Then
    \begin{enumerate}[label=(\roman*)]
        \item $x^* = s(y^*) y^*$ is a minimizer of \eqref{eq:p},
        \item for any $y \in \K$ such that $\la c, y \ra =1$, the point $x = s(y) y$ satisfies
        \BEQ
        \frac{f(x) - f(x^*)}{|f(x^*)|} \leq \frac{2}{3}\frac{ \sqrt{\rho(y)} - \sqrt{\rho(y^*)} }{\sqrt{\rho(y^*)}}
        \EEQ
    \end{enumerate}
\end{proposition}
\begin{proof} 
If $y^*$ is  minimizer of \eqref{eq:p_hom}, then it is also a minimizer of $\fhom$ on $\K$, and by \eqref{eq:minimizer_phom} the point $x^* = s(y^*)y^*$ is a minimizer of \eqref{eq:p}. Then, for $y \in \K$ such that $\la c, y\ra = 1$ we have
    \BEQ
 \begin{array}{rl}
        f(s(y) y) - f(s(y^*) y^*) = & \fhom(y) - \fhom(y^*)
        =\frac{3}{4^{4/3}} \left( \frac{-1}{\rho(y)^{1/3}} - \frac{-1}{\rho(y^*)^{1/3}} \right)\\
        \leq & \frac{2}{4^{4/3} \rho(y^*)^{5/6} } \left( \sqrt{\rho(y)} - \sqrt{\rho(y^*)} \right)
    \end{array}
 \EEQ
    where we used concavity of the function $u \mapsto -u^{-2/3}$. Then, we notice that
    \[ \frac{2}{4^{4/3} \rho(y^*)^{5/6} } = \frac{2}{3 \sqrt{\rho(y^*)} }\cdot \frac{3}{ 4^{4/3} \rho(y^*)^{1/3} } = \frac{2}{3 
    \sqrt{\rho(y^*)} } |\fhom(y^*)| \]
    which yields the result as $\fhom(y^*) = f(x^*)$.
\end{proof}

Proposition \ref{prop:hom} states that an approximative solution $x$ of \eqref{eq:p} can be obtained as $s(y)y$, where $y$ is an approximate solution of the \emph{homogenized} problem \eqref{eq:p_hom}. We chose to minimize the square root of $\rho$ in \eqref{eq:p_hom} since, as we showed in Section~\ref{s:quartic_prop}, it has favorable properties for gradient methods. Let us now apply such scheme to \eqref{eq:p_hom}.

\subsection{Fast sublinear rates for squared uniform convexity}\label{ss:convrates}
Motivated by Problem \eqref{eq:p_hom}, let us study the gradient methods for solving problems of the form
\BEQ\label{eq:p_hom_gen}
    \min_{y \in \C} g(y),
\EEQ
where, throughout this section, we assume the following:
\begin{itemize}
    \item $\C$ is a closed convex set of some Euclidean space $E$,
    \item $g : E \mapsto \reals$ is a convex function with $L$-Lipschitz continuous gradients and 
    $g >0$ on $\C$,
    \item $g^2$ satisfies the following uniform convexity property of degree 4:
    \BEQ\label{eq:unif_conv}
        g^2(x) - g^2(y) -\la \nabla (g^2)(y),x-y \ra \geq \frac{\mu^2}{4} \|x-y\|^4 \quad \forall x,y \in E.
    \EEQ    
\end{itemize}

Note that the assumptions on $g$ imply that $\mu 
\leq L$. Indeed, by Lipschitz gradient continuity,
\[
  g^2 (x) \leq \left[ g(y) + \la \nabla g(y),x-y \ra + \frac{L}{2}\|x-y\|^2 \right]^2, \quad \forall x,y \in E.
\]
Take now $x\rightarrow \infty$ with $y$ fixed. By comparing the dominant quartic term with that of the lower bound \eqref{eq:unif_conv}, we deduce that $\mu \leq L$.

In our problem of interest \eqref{eq:p_hom}, we have $\C  = \K \cap \{ y \,:\, \la y,c \ra =1 \}$, the function $g = \sqrt{\rho}$ is convex and smooth with constant $L = 6{\beta}$ (Proposition \ref{prop:quartic_form_hessian}), and $g^2 = \rho$ satisfies \eqref{eq:unif_conv} with $\mu = \sqrt{{4}/{3}} \alpha$ (Proposition \ref{prop:quartic_lower_bound}).

\begin{lemma}\label{lemma:fp_bound}
    Denoting $y^* = \argmin_{\C}g$ we have for every $y \in \C$
\BEQ\label{eq:strconv_p}
g^2(y)-g^2(y^*) \geq \frac{ \mu^2 }{4} \|y-y^*\|^4.
\EEQ
\end{lemma}
\begin{proof}
    Let $y \in \C$. \eqref{eq:unif_conv} implies
    \[ g^2(y) - g^2(y^*) - \la \nabla (g^2)(y^*),y-y^*\ra \geq \frac{ \mu^2 }{4} \|y-y^*\|^4. \]
    Since $g \geq 0$, $y^*$ is also the minimizer of $g^2$ on $\C$, and therefore $ \la \nabla (g^2)(y^*),y-y^*\ra \geq 0$.
\end{proof}

\subsubsection{Gradient descent} We first study the convergence rate of the standard gradient descent method applied to \eqref{eq:p_hom_gen}: start with $y_0 \in E$ and iterate
\BEQ\label{eq:gd}\tag{GD}
    \begin{split}
        &y_{k+1} = \argmin_{y \in \C}\, \Big\{ \la \nabla g(y_k), y-y_k \ra + \frac{L}{2}\|y-y_k\|^2\Big\},\quad k=0,1\dots,
    \end{split}
\EEQ
We define the \textit{relative accuracy} (recall that $g(y^*)>0$) at iteration $k$ as
$\delta_k \triangleq \frac{g(y_k)}{g(y^*)} - 1$.
\begin{proposition}\label{prop:convrate_gd}
    Let $\delta>0$. Then the maximal number of iterations $k$ required for \eqref{eq:gd} to reach a relative accuracy $\delta_k \leq \delta$ is at most
    \[  \frac{12L}{\mu} \left(  \ln_+(4 \delta_0) + \left( \sqrt{ 1 \over \delta} - 2 \right)_+ \right). \]
\end{proposition}

\begin{proof}
Indeed, in view of our assumptions and Lemma \ref{lemma:fp_bound}, we have
\[
\begin{array}{rcl}
g(y_{k+1}) & \leq & \min\limits_{y \in {\cal C}} \Big\{ g(y_k) + \la \nabla g(y_k), y-y_k \ra + \frac{L}{2}\|y-y_k\|^2\Big\}\\
\\
& \leq & \min\limits_{0 \leq \tau \leq 1} \Big\{ g(y_k) + \tau \la \nabla g(y_k), y^*-y_k \ra + \frac{L \tau^2}{2}\|y^*-y_k\|^2\Big\}\\
\\
& \leq & \min\limits_{0 \leq \tau \leq 1} \Big\{ g(y_k) - \tau [g(y_k) - g(y^*)] + \frac{L \tau^2 }{\mu} \sqrt{ g^2(y_k) - g^2(y^*)}\Big\}.
\end{array}
\]
The optimal value of $\tau$ in the latter problem is $\tau^* = {\mu \over 2L} \sqrt{g(y_k)- g(y^*) \over g(y_k) + g(y^*)} \leq {1 \over 2}$. Hence,
\[
\delta_{k+1} \leq \delta_k - {\mu \over 4L} \delta_k \sqrt{ g(y_k) - g(y^*) \over g(y_k) + g(y^*)} = \delta_k - {\mu \over 4L} {\delta_k^{3/2} \over \sqrt{2 + \delta_k}}.
\]
Note that the values $\{ \delta_k \}_{k \geq 0}$ decrease monotonically. Denote by $t_0$ the moment when $\delta_k \leq {1 \over 4}$ for all $k > t_0$. Then, for all $k \leq t_0$ we have $2 \leq 8 \delta_k$, and therefore
\[
\delta_{k+1} \leq \left(1 - {\mu \over 12 L} \right) \delta_k \leq \exp \Big\{ - {\mu \over 12 L} \Big\} \delta_k, \quad 0 \leq k < t_0.
\]
Hence, $t_0 \leq {12 L \over \mu} \ln_+ (4\delta_0)$. For $k \geq t_0$, we have $\delta_{k+1} \leq \delta_k - {\mu \over 6L} \delta_k^{3/2}$. Therefore,
\[
\delta_{k+1}^{-1/2} - \delta_k^{-1/2} = {\delta_k - \delta_{k+1} \over \delta_k^{1/2} \delta_{k+1}^{1/2}(\delta_k^{1/2}+\delta_{k+1}^{1/2})} \geq \; {\mu \over 12 L}.
\]
Thus, we get $\delta_k^{-1/2} \geq \delta^{-1/2}_{t_0} + {\mu  \over 12L}(k - t_0) \geq 2 + {\mu  \over 12L}(k - t_0)$.
\end{proof}

\subsubsection{Accelerated gradient descent with restarts} For a better convergence rate, consider the \textit{accelerated gradient method} (see e.g., \citep[Section 2.2.4]{Nesterov2020})
\BEQ\label{eq:fgd}\tag{AGD}
    \begin{split}
        &y_0 \in E, z_0 = y_0\\
        &y_{k+1} = \argmin_{y \in \C}\, \Big\{ \la \nabla g(z_k), y-z_k \ra + \frac{L}{2}\|y-z_k\|^2\Big\},\\
        &z_{k+1} = y_{k+1} + \frac{k}{k+3}\left( y_{k+1} - y_k \right),\\
    \end{split}
\EEQ
for $k \geq 0$. Let us write ${\rm AGD}(y_0,k)$ the output $y_k$ of this algorithm initialized at $y_0$ with $k$ iterations. Using the technique of scheduled restarts \citep{Nemirovski1983,Nesterov2020} allows to prove an improved convergence rate for our setting.

\begin{proposition}\label{prop:restart}
For $t \geq 0$ and $y_0 \in E$, 
consider the following scheme: 
    \BEQ\label{eq:fgd_restart}\tag{AGD-restart}
    \begin{split}
    &\mbox{Compute } \hat{y}_{t+1} = {\rm AGD}(y_{t}, 2^t),\\
    &\mbox{Choose } y_{t+1} = \argmin_y \big\{g(y)\,:\, y \in \{y_0,\dots,y_{t}, \hat{y}_{t+1}\}\big\},
    \end{split}
    \EEQ
Then the number of total iterations (i.e., gradient oracle calls in \eqref{eq:fgd}) needed to reach a relative accuracy less than $\delta$ is at most
\BEQ
 16 \sqrt{\frac{L}{\mu}} \frac{\left(\delta_0^{1/4} + 2^{1/4}\right) }{\delta^{1/4}}
\EEQ
\end{proposition}

\begin{proof}
Denote $\delta_t = {g(y_t) \over g(y^*)}-1$. Let us prove by induction the following bound:
\BEQ\label{eq:rec_fgd}
\delta_t \leq \frac{\delta_0}{16^{t-t0}} \quad \forall t \geq t_0,
\EEQ
where $t_0$ is the smallest integer such that
\BEQ\label{eq:def_t0} 4^{t_0}\geq  \frac{64 L}{\mu }\left( 1 + \sqrt{\frac{2}{\delta_0}}\right).\EEQ
Condition \eqref{eq:rec_fgd} is satisfied for $t=t_0$, because the outer scheme \eqref{eq:fgd_restart} is monotone by construction. Assume that it holds for some $t\geq t_0$. Then, by the known convergence result of the \eqref{eq:fgd} scheme \citep[Thm 10.34]{Beck17}, we have
    \[g(y_{t+1})-g(y^*) \leq \frac{2L\|y_{t} - y^*\|^2}{ (2^{t})^2} . \]
    Then,by Lemma \ref{lemma:fp_bound}, we get the following bound:
    \BEQ
        \begin{split}
            \delta_{{t+1}} &\leq \frac{4L}{4^t \mu  } \sqrt{ (1 + \delta_{t})^2 - 1} \leq \frac{4L}{4^t \mu   } \left( \delta_{t} + \sqrt{2\delta_{t}} \right) \leq \frac{4L}{4^t \mu   } \left( \frac{\delta_{0}}{16^{t-t_0}} + \frac{\sqrt{2\delta_{0}}}{4^{t-t_0}} \right)
        \end{split}
    \EEQ
    where the last line uses the induction hypothesis \eqref{eq:rec_fgd}. It follows that
    \BEQ \begin{split}
            \delta_{{t+1}} &\leq \frac{\delta_0}{16^{t+1-t_0}} \cdot \frac{4 \cdot 16 L}{4^t \mu   } \left( 1 + \sqrt{\frac{2}{\delta_0}} 4^{t-t_0} \right) \\ & \leq \frac{\delta_0}{16^{t+1-t_0}} \cdot \frac{64 L}{4^{t_0} \mu   } \left( 4^{t_0-t} + \sqrt{\frac{2}{\delta_0}}  \right) \\
            &\leq \frac{\delta_0}{16^{t+1-t_0}} \cdot \frac{64 L}{4^{t_0} \mu   } \left( 1 + \sqrt{\frac{2}{\delta_0}}  \right) \leq \frac{\delta_0}{16^{t+1-t_0}},
    \end{split} \EEQ
    where the last inequality follows from \eqref{eq:def_t0}, proving thus the induction. 

    Now, assume that after $T$ outer iterations, procedure \eqref{eq:fgd_restart} has reached an accuracy $\delta_T \geq \delta$ for some $\delta>0$. Hence, by \eqref{eq:rec_fgd} we have $16^T \leq 16^{t_0}\delta_0 / \delta$. Note also by definition \eqref{eq:def_t0} of $t_0$, 
    \[ 4^{t_0-1} \leq  \frac{64 L}{\mu }\left( 1 + \sqrt{\frac{2}{\delta_0}}\right). \]
    Thus, we can bound the total number of inner iterations performed by the method as
    \BEQ
        \begin{split}
            \sum_{t=0}^{T-1} 2^t = 2^{T}-1 & \leq 2^{t_0} \left(\frac{\delta_0}{\delta}\right)^{1/4} \leq 2^{t_0} \left(\frac{\delta_0}{\delta}\right)^{1/4}\\
             & \leq \sqrt{\frac{2^8 L}{\mu} \left( 1 + \sqrt{\frac{2}{\delta_0}}\right) } \left(\frac{\delta_0}{\delta}\right)^{1/4}\\
             &\leq 16 \sqrt{\frac{L}{\mu}} \left( 1 + \left(\frac{2}{\delta_0}\right)^{1/4}\right)\left(\frac{\delta_0}{\delta}\right)^{1/4}.
    \end{split}
    \EEQ
\end{proof}

\subsection{Convergence rate on the quartic problem}

After studying the complexity of methods for solving \eqref{eq:p_hom}, we get back to the original problem \eqref{eq:p}: the procedure is summarized in Algorithms \ref{algo:gd} (basic version) and \ref{algo:fgd} (accelerated version). 

\paragraph{Implementation of the projected gradient step} We describe how to compute the projected gradient step to get $y_{k+1}$ in Algorithm \ref{algo:gd}. First, note that we can decompose it as a gradient step followed with a projection step:
\begin{align} 
    \hat{y}_{k+1} &= y_k - \frac{1}{12 \beta \sqrt{ \rho(y_k)}} B^{-1} \nabla \rho(y_k),\\
    \label{subeq:proj}y_{k+1} &= \argmin_{
    \substack{
      y \in \mathcal{K}\\
      \la  y,c \ra = 1
    }
    }
    \,\,\frac{1}{2} \|y - \hat{y}_{k+1}\|^2,
\end{align}
where we recall that $B$ is the positive definite operator defining the norm.

When the original problem is unconstrained ($\K = E$), the projection is
\begin{align*} 
    y_{k+1} &= \hat{y}_{k+1} + \left( \frac{1 - \la \hat{y}_{k+1},c \ra }{ \|c\|^2_* } \right) B^{-1} c.
\end{align*}
When $\mathcal{K}$ is a cone to which the projection is easily computable, we can solve a dual of Problem \eqref{subeq:proj} which writes
\[
  \max_{t \in \reals} \min_{y \in \mathcal{K}} \Big\| y- \hat{y}_{k+1} - \frac{t}{2}B^{-1}c \Big\|^2 + t\left( 1-\la \hat{y}_{k+1},c \ra \right) - \frac{1}{2}t^2 \|c\|_*^2.
\]
The solution $t^*$ can be computed with a bisection method on $t$, which requires only a logarithmic number of projections on $\mathcal{K}$. Then, the projection is $y_{k+1} = \hat{y}_{k+1} + t^* B^{-1}c$.
\begin{algorithm}[h]
\begin{spacing}{1.4}
\caption{Homogenized-GD$(x_0, \beta, K)$}
\label{algo:gd}
\begin{algorithmic}[1]
\STATE \textbf{Initialize} $y_0\in E.$
\FOR{$k=0,1,\ldots, K-1$}
\STATE $ y_{k+1} = \argmin\, \Big\{ \frac{1}{2\sqrt{\rho(y_k)}} \la \nabla \rho(y_k), y-y_k \ra + 3{\beta}\|y-y_k\|^2 \,:\, y \in \K, \, \la c,y \ra = 1  \Big\}.$
\ENDFOR
\STATE \textbf{Return} $x_K = y_K / \left( 4\rho(y_K)\right)^{1/3}$.
\end{algorithmic}
\end{spacing}
\end{algorithm}

\begin{algorithm}[h]
\begin{spacing}{1.6}
\caption{Accelerated Homogenized-GD$(x_0, \beta, T)$}
\label{algo:fgd}
\begin{algorithmic}[1]
\STATE \textbf{Initialize} $y_0 \in E.$
\FOR{$t=0,1,\ldots, T-1$}
\STATE Set $y_t^0 = z_t^0 = y_{t-1}$
\FOR{$k=0,1,\ldots, 2^t-1$}
\STATE $ y_{t}^{k+1} = \argmin\, \Big\{ \frac{1}{2\sqrt{\rho(z_t^k)}} \la \nabla \rho(z_t^k), y-z_t^k \ra + 3{\beta}\|y-z_t^k\|^2 \,:\, y \in \K, \, \la c,y \ra = 1  \Big\}.$
\STATE $z_t^{k+1} = y_t^{k+1} + \frac{k}{k+3}\left( y_t^{k+1} - y_t^k \right)$
\ENDFOR
\STATE Set $y_t = \argmin_y \big\{ \rho(y) : y \in \{y_{t-1}, y_t^{2^t}\}\big\} $
\ENDFOR
\STATE \textbf{Return} $x_T = y_T / \left( 4\rho(y_T)\right)^{1/3}$.
\end{algorithmic}
\end{spacing}
\end{algorithm}

\begin{theorem}\label{thm:final_rate}
    Algorithm \ref{algo:gd} finds a point $x_K$ satisfying
    \BEQ
        \frac{f(x_K) - f(x^*)}{|f(x^*)|} \leq \delta
    \EEQ
    in at most
    $ \mathcal{O}\left( {\kappa}\left( \log(\frac{\delta_0}{\delta}) + \sqrt{\frac{1}{\delta}}\right)\right)$ iterations, where the initial relative accuracy is $\delta_0 = \sqrt{\frac{\rho(y_0)}{\rho(y^*)}} - 1$, with $y^*$ being the solution of the homogenized problem \eqref{eq:p_hom}, and $\kappa = \beta / \alpha$ is the quartic condition number defined by \eqref{eq:q_cond}. 
\end{theorem}
\begin{proof}
First, note that the algorithm outputs a point $y=s(y)y$ where $s$ is defined in \eqref{eq:def_tau}. Then, Proposition \ref{prop:hom} ensures
\BEQ\label{eq:recall_hom}  \frac{f(x) - f(x^*)}{|f(x^*)|} \leq \frac{2}{3}\frac{ \sqrt{\rho(y)} - \sqrt{\rho(y^*)} }{\sqrt{\rho(y^*)}}, \EEQ
where $y^*$ is a solution of the homogenized problem \eqref{eq:p_hom}. This problem satisfies all assumptions for the general class studied in Section \ref{ss:convrates}: $g = \sqrt{\rho}$ is convex and smooth with constant $L = 6{\beta}$ (Proposition \ref{prop:quartic_form_hessian}), and $g^2 = \rho$ satisfies \eqref{eq:unif_conv} with $\mu = \sqrt{\frac{4}{3}} \alpha$ (Proposition \ref{prop:quartic_lower_bound}). Algorithm \ref{algo:gd} corresponds to procedure \eqref{eq:gd} applied to \eqref{eq:p_hom}.

Proposition \ref{prop:convrate_gd} states that the number of iterations $K$ needed for finding a point $y_K$ satisfying $\sqrt{\rho(y_K)}\leq (1+\delta) \sqrt{\rho(y^*)}$ is at most
\BEQ
\frac{12L}{\mu} \left(  \ln_+(4 \delta_0) + \left( \sqrt{ 1 \over \delta} - 2 \right)_+ \right).
\EEQ
Recalling that $L = 6 {\beta},\mu = \sqrt{\frac{4}{3}}\alpha$, we have $\frac{L}{\mu} = \sqrt{27} \kappa$, which, combined with \eqref{eq:recall_hom}, allows to conclude.
\end{proof}

A similar reasoning with Proposition \ref{prop:restart} justifies the accelerated variant.

\begin{theorem}\label{thm:final_rate_acc}
    Accelerated homogenized gradient descent (Algorithm \ref{algo:fgd}) finds a point $x_K$ satisfying
    \BEQ
        \frac{f(x_K) - f(x^*)}{|f(x^*)|} \leq \delta
    \EEQ
    in at most
    $\mathcal{O}\left( \frac{\sqrt{\kappa}}{\delta^{1/4}} \big( 1 + \delta_0^{1/4}\big)\right)$ projected gradient steps, where $\delta_0 = \sqrt{\rho(y_0) \over \rho(y^*)}-1$.
\end{theorem}

\section{Optimal preconditioning}\label{s:precond}

We showed that the efficiency of our methods crucially depends on the {quartic condition number}. In this section, we propose to search for a norm-inducing operator $B$, or \textit{preconditioner}, such that the condition number of $\rho$ with respect to the Euclidean norm $\|\cdot\|_B$ is small.

We focus on the particular case when $\rho$ corresponds to the quartic part of a quadratic inverse problem \eqref{eq:noncvx_p} with positive semidefinite operators:
\BEQ\label{eq:rho_form1}
    \rho(x) = \sum_{i=1}^m \la x, B_i x \ra^2.
\EEQ
For simplicity, we choose $E = \reals^n$, and let the following standing assumptions hold. 
\begin{assumption}\label{ass:b}
    The $m$-uple of operators $\mathcal{B} = (B_1,\dots B_m) \in (\symm^n)^m$ is such that
\begin{enumerate}[label=(\roman*)]
    \item $m \geq n$,
    \item $B_1, \dots B_m$ are nonzero positive semidefinite matrices,
    \item each $B_i$ has rank at most $r$, where $1 \leq r \leq n$,
    \item the matrix $\sum_{i=1}^m B_i$ is of full rank.
\end{enumerate}
\end{assumption}
These assumptions are satisfied for the phase retrieval problem \eqref{eq:phase_ret}, where $r=2$, and distance matrix completion, for which $r=3$ in the majority of applications\footnote{The full rank assumption holds generically for the phase retrieval with $m$ high enough, and is satisfied for distance matrix completion if the graph is connected.}. 

We start in Section \ref{ss:levg} by addressing a more general goal. That is a quadratic approximation of the function $x\mapsto \left(\sum_{i=1}^m\la B_i x, x\ra^{p}\right)^{1/p}$ for some general power $p > 1$. We show that there exists an operator $B_*$ for which the approximation ratio of this function with respect to $\|\cdot\|^2_{B_*}$ is $n^{\frac{p-1}{p}}$, and that it can be efficiently computed by a fixed-point iteration method described in Section \ref{ss:algo_fp}. 


\subsection{Existence and characterization of the optimal preconditioner}\label{ss:levg}
Let us fix a power $p >1$, and write its conjugate power $q$ as $\frac{1}{p}+\frac{1}{q}=1$. For a given $m$-uple of matrices $\mathcal{B} = (B_1,\dots B_m)$ satisfying the assumptions above, we define the function
\[ W_{\mathcal{B},p}(x) = \sum_{i=1}^m\la B_i x, x\ra^{p},\]
for which the quartic case corresponds to $p=2$. Consider the following problem.
\vspace{1ex}
\begin{center}
\textit{
How good can be a quadratic approximation of the function $W_{\mathcal{B},p}^{1/p}(\cdot)$?}
\end{center}
\vspace{1ex}
Let us look at quadratic function $\la B(\tau)\cdot,\cdot \ra$ with an operator $B$ of the following form:  
\BEQ\label{eq:def_btau} 
B(\tau) = \sum_{i=1}^m \tau_i B_i, \quad \tau \in \reals^m_{++}.
\EEQ
Since $\tau > 0$, Assumption \ref{ass:b} implies that $B(\tau) \succ 0$. It induces therefore a norm $\|x\|_{B(\tau)} = \la B(\tau)x, x \ra^{\frac{1}{2}}$. Let us look how good is it for approximating $W_{\mathcal{B},p}$.  

\subsubsection{Comparing $W_{\mathcal{B},p}(x)$ and $\|x\|_{B(\tau)}$} 
Our analysis relies on the quantities
\BEQ\label{eq:levg_def}
    \ell_i \left(\tau;\mathcal{B}\right) =  \Tr \left[\tau_i B(\tau)^{-1}B_i \right],
\EEQ
defined for $i = 1\dots m$ and $\tau\in \reals^m_{++}$.
The $\ell_i(\tau;\mathcal{B})$ measures the importance of the matrix $\tau_i B_i$ in forming the whole $B(\tau)$. If the $B_i$'s are of rank 1, they correspond to the classical notion of \textit{leverage scores} used in regression analysis and randomized Linear Algebra \citep{MAHONEY2012,Martinsson2020randomized}. 

\begin{lemma}\label{lemma:levg_score}
    Under Assumption \ref{ass:b}, we have for any $\tau \in \reals^{m}_{++}$ and $i = 1\dots m$
    \begin{enumerate}[label=(\roman*)]
        \item \label{elem:l1} $\ell_i(\tau; \mathcal{B}) \in [0,{\rm rank}(B_i) ]$,
        \item \label{elem:l2} $\sum_{j=1}^m \ell_j(\tau; \mathcal{B}) = n$,
        \item \label{elem:l3} $\tau_i \la x, B_i x \ra \leq \ell_i(\tau; \mathcal{B}) \, \|x\|_{B(\tau)}^2 $ for all $x\in\reals^n$.
    \end{enumerate}
\end{lemma}
\begin{proof}
Since all $B_i \succeq 0$, we have $\tau_i B_i \preceq B(\tau)$ and therefore the matrix
\[ Z_i = \tau_i B(\tau)^{-\frac{1}{2}} B_i B(\tau)^{-\frac{1}{2}} \]
satisfies $0 \preceq Z_i \preceq I_n$. This implies 
$0\leq \Tr\left[ Z_i\right] \leq {\rm rank}(Z_i) = {\rm rank}(B_i)$,
proving the Item (i) since $\ell_i(\tau; \mathcal{B}) = \Tr\left[Z_i\right]$.
Item (ii) follows from the equality 
\[\sum_{i=1}^m \Tr\left[ \tau_i B(\tau)^{-1} B_i\right] = \Tr\left[ B(\tau)^{-1}B(\tau) \right] = n.\]
Finally, to prove (iii) we write
\BEQ\begin{split}
    \la B_i x, x \ra &= \| B_i^{\frac{1}{2}} x \|_2^2 = \| B_i^{\frac{1}{2}} B(\tau)^{-\frac{1}{2}} B(\tau)^{\frac{1}{2}} x \|_2^2 \\
    &\leq \| B_i^{\frac{1}{2}} B(\tau)^{-\frac{1}{2}}\|^2_F \, \| B(\tau)^{\frac{1}{2}} x \|_2^2
    = \Tr\left[ B(\tau)^{-\frac{1}{2}} B_i B(\tau)^{-\frac{1}{2}} \right] \la Bx, x \ra\\
    &= \tau_i^{-1} \ell_i(\tau; \mathcal{B}) \,\|x\|_{B(\tau)}^2.
\end{split}
\EEQ
\end{proof}

 Using the quantities $\ell_i(\tau; \mathcal{B})$, we can compare $\|x\|^2_{B(\tau)}$ with $W_{\mathcal{B},p}(x)$.

\begin{proposition}\label{prop:levg_cond}

For any $\tau \in \reals^m_{++}$ and $x \in \reals^n$ we have
\BEQ \label{eq:cond_ineq}
  \left( \frac{1}{{\|\tau\|_{q}}} \right) \| x\|^2_{B(\tau)} \leq W_{\mathcal{B},p}(x)^{\frac{1}{p}} \leq \left( \max_{i=1\dots m} \frac{ \ell_i(\tau; \mathcal{B})^{ \frac{1}{ q} } }{ {\tau_i} } \right) \|x\|^2_{B(\tau)}.
\EEQ
\end{proposition}
\begin{proof}
    The lower bound follows from Hölder's inequality, as $\frac{1}{p}+\frac{1}{q}= 1$:
    \[\|x\|_{B(\tau)}^2 = \sum_{i=1}^m \tau_i \la x,B_i x\ra \leq \left( \sum_{i=1}^m \tau_i^q \right)^{\frac{1}{q}} \left( \sum_{i=1}^m \la x,B_i x\ra^{p} \right)^{\frac{1}{p}} = \|\tau\|_q\, W_{\mathcal{B},p}(x)^\frac{1}{p}. \]
    For the upper bound, we write
    \BEQ \label{eq:bound_axp}
    \begin{split}
    W_{\mathcal{B},p}(x) &= \sum_{i=1}^m \la x,B_i x \ra^{p} = \sum_{i=1}^m \frac{ \left( \tau_i \la x,B_i x \ra \right)^{p-1} }{\tau_i^p}\, \tau_i \la x,B_ix \ra    \\
     &\leq \left(\max_{i=1\dots m} \frac{ \left( \tau_i \la x,B_ix\ra \right)^{p-1} }{\tau_i^p} \right) \|x\|^2_{B(\tau)} 
     \leq \left(\max_{i=1\dots m} \frac{ \ell_i(\tau; \mathcal{B})^{p-1} }{\tau_i^p} \right) \|x\|^{2p}_{B(\tau)},
    \end{split}
    \EEQ
    where the last inequality follows from Lemma \ref{lemma:levg_score}\ref{elem:l3}.
\end{proof}
\subsubsection{Existence of $n^{\frac{p-1}{p}}$-quadratic approximation} Proposition \ref{prop:levg_cond} gives an upper bound for the condition number of $W_{\mathcal{B},p}^{1/p}$ with respect to $\|\cdot\|_{B(\tau)}$. We now seek for the d coefficients $\tau$ that minimize this upper bound. 
Since the condition number is invariant by rescaling of $\tau$, we can choose to impose the condition $\|\tau \|_q=1$.

Let us show that the factor in the right-hand side of \eqref{eq:cond_ineq} is at least $n^{\frac{1}{q}}$.
\begin{proposition} For every $\tau \in \reals^{m}_{++}$ such that $\|\tau\|_q = 1$, we have
\[ \max_{i=1\dots m} \frac{ \ell_i(\tau; \mathcal{B})^{ \frac{1}{ q} } }{ {\tau_i} } \geq n^{\frac{1}{q}}.\]
\end{proposition}
\begin{proof}
    If, on the contrary, we have $\ell_i(\tau; \mathcal{B}) < n \tau_i^q$ for every $i \in \{1\dots m\}$, we reach a contradiction as $\sum_{i=1}^m \ell_i(\tau; \mathcal{B}) = n$ by Lemma \ref{lemma:levg_score}\ref{elem:l2} and $\sum_{i=1}^m \tau_i^q = 1$.
\end{proof}

It appears that there is a unique vector $\tau^*$ which attains this minimal value.

\begin{proposition}\label{prop:existence_bstar}
    There exists a unique vector $\tau^* \in \reals^m_{++}$ such that
    \BEQ\label{eq:taustar_eq} \frac{ \ell_i(\tau^*; \mathcal{B})^{ \frac{1}{ q} } }{ {\tau_i^*} } = n^{\frac{1}{q}}, \quad i = 1\dots m.\EEQ
    This vector satisfies $\|\tau^*\|_q = 1$. It is a unique solution to the minimization problem
    \BEQ\label{eq:v_exact}
        \min_{ \tau \in \reals^m } V(\tau) = - \ln \det B(\tau) + \frac{n}{q} \|\tau\|_q^q.
    \EEQ
\end{proposition}
\begin{proof}
    As $q>1$, the function $V$ is strictly convex and differentiable on the set
    \[ \mathcal{D} = \{\tau \in \reals^m\,:\, B(\tau) \succ 0\}. \]
    Since $V(\tau) \rightarrow \infty$ as $\tau \rightarrow \partial \mathcal{D}$ (the boundary of $\mathcal{D}$) or $\|\tau\|_q \rightarrow \infty$, Problem \eqref{eq:v_exact} admits a unique minimizer $\tau^*$, which belongs to $\mathcal{D}$.

    As the gradient of function $B \mapsto \ln\det B$ is $B^{-1}$, the first-order optimality conditions write
    \BEQ\label{eq:LHS} \Tr \left[B(\tau^*)^{-1} B_i \right] = n \,{\rm sign}(\tau_i^*) |\tau^*_i|^{q-1}, \quad i = 1\dots m. \EEQ
    As we assumed the $B_i$'s to be nonzero and $\tau^* \in \mathcal{D}$, we have $\Tr \left[B(\tau^*)^{-1} B_i \right] > 0$ and therefore $\tau^*_i > 0$ for every $1\leq i \leq m$. Now, as the left-hand side of \eqref{eq:LHS} is $\ell_i(\tau^*; \mathcal{B})/\tau_i^*$, it follows that
    \BEQ\label{eq:tau_eq}
        \ell_i(\tau^*; \mathcal{B}) = n (\tau_i^*)^q,    \quad i = 1\dots m,
    \EEQ
    which proves \eqref{eq:taustar_eq}. Summing up these equations and using Lemma \ref{lemma:levg_score}\ref{elem:l2}, we get
    \[\|\tau^*\|_q^q = \frac{1}{n}\sum_{i=1}^n \ell_i(\tau^*; \mathcal{B}) = 1.\]
\end{proof}

In the case $r=1$, vector $\tau^*$ corresponds exactly to the \textit{Lewis weights} \citep{Lewis1978FiniteDS,Cohen2015lewis} used in $\ell_p$ subspace sampling. We allow operators $B_1,\dots B_m$ to have a higher rank, and show that the same construction can be used for building an efficient quadratic approximation for $W_{\mathcal{B},p}^{1/p}(\cdot)$.
We summarize the discussion in the following theorem.
\begin{theorem}[Optimal quadratic approximation of $W_{\mathcal{B},p}$]\label{thm:cond_taustar}
    For every $\mathcal{B} = (B_1,\dots B_m)$ satisfying Assumption \ref{ass:b}, there exists a matrix ${B_*\in \symm^n_{++}}$ such that for every $x \in \reals^n$
     \BEQ \label{eq:cond_bstar}
    \| x\|^2_{B_*} \leq \left(\sum_{i=1}^m\la B_i x, x\ra^{p} \right)^{\frac{1}{p}} \leq n^{1-\frac{1}{p}} \|x\|^2_{B_*}
    \EEQ 
\end{theorem}
\begin{proof}
    This follows by taking $B_*= B(\tau^*)$, where $\tau^*$ is given by Proposition \ref{prop:existence_bstar}, and applying Proposition \ref{prop:levg_cond}.
\end{proof}

\subsubsection{Optimality for general $\mathcal{B}$} One can ask if the constant $n^{1-1/p}$ in \eqref{eq:cond_bstar} is optimal. This is true for a \textit{generic} $\mathcal{B}$, although it might be better for specific choices. 

Indeed, let us choose $m = n$ and $B_i = e_i e_i^T$ for $i = 1\dots n$, where $\{e_i\}_{1\leq n}$ denote the coordinate vectors of $\reals^n$. We have
\[ W_{\mathcal{B},p}(x) = \sum_{i=1}^n x_i^{2p} = \|x\|_{2p}^{2p}. \]
A symmetry argument implies that the best quadratic approximations $\la B\cdot, \cdot \ra$ are isotropic, i.e. $B \propto I_n$. For such $B$, the approximation ratio is $n^{1-1/p}$ as we have
\[ n^{\frac{1}{p}-1}\|x\|_2^2 \leq \|x\|^{2}_{2p}  \leq \|x\|_{2}^2, \]
which is tight (take $x = (1\dots 1)$ for matching the lower bound and $x = e_1$ for the upper bound). This shows that the bound \eqref{eq:cond_bstar} is not improvable for general $\mathcal{B}$.

\subsection{Fixed-point algorithm for computing the optimal weights}
\label{ss:algo_fp} 
We now show how the coefficients $\tau^*$ can be computed using a fixed point algorithm when $p<2$.  This process converges only for $p < 2$, but we will show in Section \ref{sss:p2} that we can also deal with the case $p=2$ (which corresponds to the quartic applications) by using the fixed-point method with a surrogate power $p' < 2$.

\subsubsection{Case $p < 2$}

The equations \eqref{eq:taustar_eq}, which characterize $\tau^*$, are as follows:
\BEQ\label{eq:fixed_point}
    \ell_i(\tau_i;\mathcal{B}) = n \tau_i^{\frac{p}{p-1}}, \quad i = 1\dots m.
\EEQ
By definition \eqref{eq:levg_def} of $\ell_i$, \eqref{eq:fixed_point} can be written equivalently as the fixed-point equation
\[ \tau_i = \left( \frac{1}{n} \Tr\left[B(\tau)^{-1}B_i \right] \right)^{p-1},\quad i = 1 \dots m. \]
Based on this observation, we propose in Algorithm \ref{algo:fp} a fixed-point method for approximating $\tau^*$. This method was proposed initially by \citep{Cohen2015lewis} for computing the Lewis weights with $r=1$, and we generalize it to possibly larger rank $r$.

\begin{algorithm}[h!]
\begin{spacing}{1.6}
\caption{Computing approximate optimal weights of order $p \in (1,2)$}
\label{algo:fp}
\begin{algorithmic}[1]
\STATE \textbf{Input} $p\in (1,2)$, $\mathcal{B} = (B_1,\dots B_m)$, precision $\epsilon>0$.
\STATE \textbf{Initialize} $\tau^{(0)} = (m^{\frac{1}{p}-1},\dots m^{\frac{1}{p}-1})$.
\STATE \textbf{for} $k = 0,1,\dots$ do\\
\vspace{1ex}
\STATE \quad  $\tau^{(k+1)}_i = \left( \frac{1}{n} \Tr\left[ \left( \sum_{j=1}^m \tau^{(k)}_j B_j\right)^{-1}B_i \right] \right)^{p-1} \quad \text{for } i = 1\dots m,$
\vspace{1ex}
\STATE \textbf{until} $D_L(\tau^{(k+1)}, \tau^{(k)}) \leq {2-p \over p-1} \epsilon$.
\end{algorithmic}
\end{spacing}
\end{algorithm}

\paragraph{Convergence analysis of Algorithm \ref{algo:fp}.} Define the fixed-point operator $T_p$ as 
\BEQ\label{eq:def_T}
    \left[T_p(\tau)\right]_i =\left( \frac{1}{n} \Tr \left[B(\tau)^{-1}B_i \right] \right)^{p-1}, \quad \tau\in\reals^m_{++}, \quad i = 1\dots m.
\EEQ
Note that if $\tau >0$, then $B(\tau)$ is positive definite and $T_p(\tau) >0$. Therefore, $T_p$ is a map from $\reals^m_{++}$ to itself.
We define the log-infinity distance $\dlog(u,v)$ for $u,v \in \reals^m_{++}$:
\[ \dlog(u,v) = \max_{i=1\dots m} |\ln u_i - \ln v_i|.\]
We show that, for $p<2$, $T_p$ is a contraction with respect to the distance $\dlog$.
\begin{lemma} \label{lemma:contraction}
For every $\tau,\tilde{\tau}\in\reals^m_{++}$ we have
$\dlog\Big( T_p(\tau), T_p(\tilde{\tau}) \Big) \leq (p-1)\, \dlog(\tau,\tilde{\tau}).$
\end{lemma}
\begin{proof}    
    Let $\tau, \tilde{\tau} \in \reals^m_{++}$, and let $\Delta = \exp \left[ \dlog(\tau,\tilde{\tau})\right]$ so that
    \[ \frac{1}{\Delta} \tilde{\tau}_i \leq \tau_i \leq \Delta \tilde{\tau}_i, \quad i = 1\dots m. \]
    This implies that $\Delta^{-1} \tilde{\tau}_i B_i \preceq \tau_i B_i \preceq \Delta \tilde{\tau}_i B_i$ for every $i=1\dots m$ and therefore
    \[ \Delta^{-1} B(\tilde{\tau})^{-1} \preceq B(\tau)^{-1} \preceq \Delta B(\tilde{\tau})^{-1}.\]
    Using the expression \eqref{eq:def_T} for $T_p$ yields, recalling that $p>1$,
    \[ \Delta^{-(p-1)}\left[T_p(\tilde{\tau})\right]_i \,\leq\, \left[T_p(\tau)\right]_i\,\leq\,  \Delta^{p-1} \left[T_p(\tilde{\tau})\right]_i, \quad i = 1\dots m, \]
    which implies $\dlog\left( T_p(\tau), T_p(\tilde{\tau}) \right)  \leq (p-1) \ln \Delta$.
\end{proof}
The contraction property allows to establish the convergence rate to the fixed point.
\begin{proposition}\label{prop:convrate}
    Assume that $1<p<2$. The iterates of the fixed-point process 
    \BEQ
    \begin{cases}
      &\tau^{(0)} = \left(m^{\frac{1}{p}-1},\dots, m^{\frac{1}{p}-1}\right),\\
      &\tau^{(k+1)} = T_p(\tau^{(k)}),\quad  k = 0, 1 \dots
    \end{cases}
    \EEQ
    converge to the solution $\tau^*$ of \eqref{eq:taustar_eq} and satisfy for $k\geq 1$
    \BEQ
      \dlog\left( \tau^{(k)},\tau^{*} \right) \leq \left(1 - {1 \over p} \right)^{k+1} \ln(rm).
    \EEQ
\end{proposition}
\begin{proof} As, by Lemma \ref{lemma:contraction}, $T_p$ is a contraction for $p < 2$, and $\tau^*$ is the unique vector satisfying $T_p(\tau^*) = \tau^*$, we can apply the Banach fixed point theorem in the space $(\reals^m_{++},\dlog)$, which proves convergence. Moreover, for $k\geq 1$, we have 
  \BEQ\label{eq:rate_abstr}\dlog\left( \tau^{(k)},\tau^{*} \right) \leq (p-1)^{k-1} \dlog\left( \tau^{(1)},\tau^* \right),
  \EEQ 
  To bound the distance of $\tau^{(1)}$ to the optimum, let us show that $B(\tau^{(0)})$ and $B(\tau^*)$ are comparable.
    Since $\|\tau^{(0)}\|_q = 1$, Proposition \ref{prop:levg_cond} as applied to $\tau^{(0)}$ yields
  \BEQ\label{eq:bound_b0ab}
  \| x\|^2_{B(\tau^{(0)})} \leq W_{\mathcal{B},p}(x)^{\frac{1}{p}} \leq \left( \max_{i=1\dots m} m\,\ell_i(\tau^{(0)}; \mathcal{B})  \right)^{1-\frac{1}{p}} \|x\|^2_{B(\tau^{(0)})}  \EEQ
  for every $x\in \reals^n$. By Lemma \ref{lemma:levg_score}\ref{elem:l1}, we have 
  \BEQ\label{eq:bound_mli}\max_{i=1\dots m} \,\ell_i(\tau^{(0)}; \mathcal{B}) \leq \max_{i=1\dots m} \mbox{rank}(B_i) = r.
  \EEQ
    Similarly, according to Theorem \ref{thm:cond_taustar}, $B(\tau^*)$ satisfies
  \BEQ\label{eq:bound_bsab}
      \| x\|^2_{B(\tau^*)} \leq W_{\mathcal{B},p}(x)^{\frac{1}{p}} \leq n^{1-\frac{1}{p}} \|x\|^2_{B(\tau^*)}.
  \EEQ
  Combining \eqref{eq:bound_b0ab}, \eqref{eq:bound_mli} and \eqref{eq:bound_bsab}, we get
  \BEQ
       (rm)^{-(1-\frac{1}{p})} B(\tau^*) \preceq B(\tau^{(0)}) \preceq n^{1-\frac{1}{p}} B(\tau^*).
  \EEQ
  This implies that for every $i = 1\dots m$,
  \[ n^{- \frac{(p-1)^2}{p} }\left[ T_p(\tau^{*})\right]_i \leq  \left[ T_p(\tau^{(0)})\right]_i \leq (rm)^{ \frac{(p-1)^2}{p} } \left[ T_p(\tau^*)\right]_i \]
  and therefore
  \[ \dlog\big( T_p(\tau^{(1)}), \tau^* \big)  \leq p^{-1}(p-1)^2 \ln \left[ \max\left(n, rm\right) \right] \leq  p^{-1} (p-1)^2 \ln(rm),\]
  as we assumed $m \geq n$. Combining with \eqref{eq:rate_abstr} achieves the result.

\end{proof}
 This result allows us to estimate the number of iterations needed to compute an approximate solution to \eqref{eq:taustar_eq}, and therefore the complexity of Algorithm \ref{algo:fp}.
\begin{proposition}\label{prop:convrate_algo_stop}
    Assume that $1<p < 2$ and let $\epsilon>0$. Then, Algorithm \ref{algo:fp} with precision $\epsilon$ performs at most
\[ {p-1 \over 2-p} \left[ \ln {\ln(rm) \over (2-p)\epsilon} \right]_+    
\]
    iterations, and its output $\overline{\tau}$ satisfies
    \BEQ\label{eq:Capprox}
    e^{-\epsilon}\,  \|x\|^2_{B(\overline{\tau})} \leq \left(\sum_{i=1}^m\la B_i x, x\ra^{p} \right)^{\frac{1}{p}} \leq e^{\epsilon}\,n^{1-\frac{1}{p}} \|x\|^2_{B(\overline{\tau})}.
    \EEQ
\end{proposition}
\begin{proof}
  By Proposition \ref{prop:convrate}, we can bound the distance between successive iterates:
  \[ \dlog\left( \tau^{(k)},\tau^{(k-1)} \right)  \leq \dlog\left( \tau^{(k)},\tau^{*} \right) + \dlog\left( \tau^{(k-1)},\tau^{*} \right) \leq (p-1)^k \ln(rm). \]
  Since Algorithm \ref{algo:fp} stops as soon as $ \dlog\left( \tau^{(k)},\tau^{(k-1)} \right) \leq \frac{2-p}{p-1} \epsilon $, the number $k$ of the iterations performed satisfies the bound
$(p-1)^{k+1} \ln (rm) \leq (2-p) \epsilon$. Hence
\[
{\ln(rm) \over (2-p)\epsilon} \leq \left({1 \over p-1}\right)^{k+1} \leq \exp\left\{ {(k+1) (2-p)\over p-1} \right\}.
\]

  To analyse the approximation quality of the last iterate $\overline{\tau} = \tau^{(K)}$, we use the contraction property of $T_p$ and the stopping criterion:
  \BEQ\begin{split}
   \dlog\left( \tau^{(K)}, \tau^* \right) &\leq \dlog\left( \tau^{(K+1)}, \tau^{(K)} \right) + \dlog\left( \tau^{(K+1)}, \tau^* \right) \\
   & \leq (p-1) \dlog\left( \tau^{(K)}, \tau^{(K-1)} \right) + (p-1) \dlog\left( \tau^{(K)}, \tau^* \right)\\
   &\leq (2-p) \epsilon + (p-1) \dlog\left( \tau^{(K)}, \tau^* \right).
 \end{split}
   \EEQ
   Hence, $e^{-\epsilon} B(\tau^*) \preceq B(\tau^{(K)}) \preceq e^{\epsilon} B(\tau^*)$,
and using the property of $B(\tau^*)$ from Theorem \ref{thm:cond_taustar}, we get \eqref{eq:Capprox}.
\end{proof}

\paragraph{Computational complexity.} In practice, we usually have an access to a compact representation of $B_1\dots B_m$ as 
$B_i = U_i U_i^T$,  
where $U_i \in \reals^{n \times r}$ (in our examples, such representation is explicit). Then, the quantity $\ell_i(\tau)$ can be computed as
\[ \ell_i(\tau) = \tau_i \Tr\left[ U_i^T B(\tau)^{-1}U_i  \right],\]
which takes $\mathcal{O}(n^2 r)$ operations for each $i = 1\dots m$. Forming and inverting the matrix $B(\tau)$ requires $\mathcal{O}(mn^2)$ operations. Therefore, the total complexity of one iteration of Algorithm \ref{algo:fp} is $\mathcal{O}( mrn^2 )$, and we need to perform $\mathcal{O}\left(\ln (rm) / (2-p)\right)$ such iterations as we are satisfied with a constant factor approximation of the optimal preconditioner. 

\subsubsection{Quartic case ($p=2$)}\label{sss:p2}
\newcommand{\lnmn}{\ln\frac{m}{n}}
\newcommand{\overp}{p'}
The fixed-point method from the previous section converges only for $p < 2$. However, we are particularly interested in computing the weights $\tau^*$ for $p = 2$, which corresponds to preconditioning of quartic functions. In this case, let us show that a constant factor approximation of the optimal $\sqrt{n}$ ratio can be obtained by the previous method with a well-chosen surrogate power ${\overp} < 2$.

\begin{corollary}\label{cor:pequal2}
    Let $\omega > 1$. The Algorithm \ref{algo:fp} with precision $\epsilon > 0$ and power 
    \[
      \overp = \frac{2 \lnmn }{ \lnmn + 2 \ln \omega}
    \]
performs at most $\mathcal{O}\left( \left[ \ln \ln (rm) - \ln\ln \omega- \ln \epsilon \right]_+ \left( {\lnmn} / \ln \omega \right)  \right)$ iterations, and the output $\overline{\tau}$ satisfies
    \BEQ
    e^{-\epsilon} \left( \frac{1}{m^{ \frac{1}{\overp} - \frac{1}{2}}}\|x\|^2_{B(\overline{\tau})} \right) \leq \left(\sum_{i=1}^m\la B_i x, x\ra^{2} \right)^{\frac{1}{2}} \leq e^{\epsilon}\omega \, \sqrt{n}  \left( \frac{1}{m^{ \frac{1}{\overp} - \frac{1}{2}}} \|x\|^2_{B(\overline{\tau})} \right).
    \EEQ
\end{corollary}
\begin{proof}
    As $ \overp < 2$, the classical inequality between $\ell_2$ and $\ell_{\overp}$ norms yields
    \[ \frac{1}{m^{\frac{1}{\overp} - \frac{1}{2}}} W_{\mathcal{B},\overp}(x)^{\frac{1}{\overp}} \leq W_{\mathcal{B},2}(x)^{\frac{1}{2}} \leq W_{\mathcal{B},\overp}(x)^{\frac{1}{\overp}}.\]
    Combining with inequality \eqref{eq:Capprox} from Proposition \ref{prop:convrate_algo_stop} gives
    \BEQ
    \frac{1}{m^{\frac{1}{\overp} - \frac{1}{2}}}\, e^{-\epsilon} \|x\|^2_{B(\overline{\tau})} \leq W_{\mathcal{B},2}(x)^{\frac{1}{2}} \leq e^{\epsilon}n^{1-\frac{1}{\overp}} \|x\|^2_{B(\overline{\tau})},
    \EEQ
    and we conclude by noticing that $\overp$ was chosen so that $\left(\frac{m}{n}\right)^{\frac{1}{\overp} - \frac{1}{2}} = \omega$. To estimate the total number of iterations, we use Proposition \ref{prop:convrate_algo_stop} again and write
    \[ \frac{1}{2-\overp} = \frac{1}{4} + \frac{\ln \frac{m}{n}}{4\ln\omega} = \mathcal{O}\left(\frac{\ln \frac{m}{n} }{ \ln \omega}\right), \]
    and $-\ln(2-p) = \mathcal{O}\left( \ln\ln \frac{m}{n} - \ln\ln \omega \right)= \mathcal{O}\left( \ln\ln {rm} - \ln\ln \omega \right)$.
\end{proof}

Note that, for practical purposes, a good preconditioner has the approximation ratio in a constant factor away from the optimal $\sqrt{n}$. Note that Corollary \ref{cor:pequal2} states that we can find such a matrix in $\mathcal{O}\left( \ln \frac{m}{n} \ln\ln(rm) \right)$ iterations of Algorithm \ref{algo:fp}.

\subsection{Application to quartic problems}\label{ss:cond_examples}

We now detail the consequences of the previous results on the preconditioning of quartics functions of the form \eqref{eq:rho_form1}.

\paragraph{Improvement of $B_*$ over $B(\tau^{(0)})$.} 
Computing the optimal parameters $\tau^*$ requires running Algorithm \ref{algo:fp} and might be costly. To evaluate the benefits, let us compare with the simpler \textit{uniform} choice $\tau^{(0)} = (m^{-\frac{1}{2}},\dots m^{-\frac{1}{2}})$, which corresponds to the inital point of the method. Writing $ B^{(0)} = m^{-\frac{1}{2}} \sum_{i=1}^m B_i$ the corresponding preconditioner, Proposition \ref{prop:levg_cond} yields

    \BEQ\label{eq:cond_B0}  \|x\|^2_{B^{(0)}} \leq \rho(x)^{1/2} \leq  \left(\max_{i=1 \dots m} m\,\ell_i( \tau^{(0)};\mathcal{B} )\right)^{1/2}  \|x\|^2_{B^{(0)}} = \sqrt{m \gamma(\mathcal{B})}\,  \|x\|^2_{B^{(0)}},
    \EEQ
where we define $\gamma(\mathcal{B})$ to be the \textit{coherence} of $\mathcal{B} = (B_1,\dots,B_m)$:
\BEQ\label{eq:def_coherence} \gamma(\mathcal{B}) = \max_{i=1\dots m} \ell_i(\tau^{(0)};\mathcal{B}) = \max_{i=1\dots m} \Tr\left[ \left(\sum_{i=1}^m B_i\right)^{-1} B_i \right],\EEQ
The condition number of $\rho$ with respect to the norm induced by $B^{(0)}$ is therefore upper bounded by $\sqrt{m \gamma(\mathcal{B})}$.
By Lemma \ref{lemma:levg_score} we have $\frac{n}{m}\leq \gamma(\mathcal{B})\leq r$, therefore
\BEQ \label{eq:ineq_kappa0} \underbrace{\sqrt{n}}_{\small \mbox{condition number w.r.t. } B_*} \leq \underbrace{\sqrt{m \gamma(\mathcal{B})}}_{\small\mbox{condition number w.r.t. } B^{(0)}} \leq \sqrt{mr}.\EEQ
Therefore, if the coherence of $\mathcal{B}$ is small, i.e. close to $n/m$, the approximation quality of $\tau^{(0)}$ and $\tau^*$ are similar. They become different if $\gamma(\mathcal{B})$ is high, and if $mr \gg n$. This is the situation where one could benefit from using $B(\tau^*)$ instead of $B^{(0)}$.

Note that if the operators $B_i$ are of rank 1 ($r=1$) with $B_i = a_i a_i^T$, then $\gamma(\mathcal{B})$ coincides with the \textbf{matrix coherence} of $[a_1 \dots a_m]^T$, as defined in statistics and numerical Linear Algebra \citep{MAHONEY2012}.

\paragraph{Preconditioning for distance matrix completion.}
Recall that for Problem \eqref{eq:DMC}, the quartic function writes for $X\in \reals^{N\times r}$
\BEQ
\rho(X) = \sum_{(i,j) \in \Omega} \|x_i-x_j\|_2^4  +\frac{1}{N^2} \|\sum_{i=1}^N x_i\|_2^4,
\EEQ
where $x_1\dots x_N$ are the rows of $X$ and $\Omega \subset \{1\dots N\}^2$.
We can write $\rho$ in the form \eqref{eq:rho_form1} as $\rho(X) = \sum_{(i,j) \in \Omega} \la B_{ij} (X),X \ra^2 + \frac{1}{N}\la \mathbf{1}_N\mathbf{1}_N^TX,X \ra$ where the operators $B_{ij}:\reals^{N \times r} \rightarrow \reals^{N \times r}$ induce the quadratic forms $\la B_{ij}(X),X\ra = \|x_i-x_j\|_2^2$.
The norm induced by the \textit{uniform preconditioner} $B^{(0)}$ writes
\begin{align*}
  \|X\|_{B^{(0)}}^2 &= \frac{1}{\sqrt{|\Omega|+1}} \left(\sum_{(i,j)\in \Omega} \| x_{i}-x_{j} \|^2 + \big\|
  \sum_{i=1}^N x_i
  \big\|^2 \right)\\
  &= \frac{1}{\sqrt{|\Omega|+1}} \sum_{k=1}^r \la \left(\mathcal{L}(\Omega) + \mathbf{1}_N\mathbf{1}_N^T\right) x_{\cdot k},x_{\cdot k} \ra,
\end{align*}
where $x_{\cdot k}$ is the $k$-th column of $X$ and $\mathcal{L}(\Omega) \in \reals^{N\times N}$ is the Laplacian matrix of the graph with vertices $\{1\dots N\}$ and edge set $\Omega$ \cite{chung1997spectral}. Then, one can show that the coherence of $\mathcal{B}$ as defined by \eqref{eq:def_coherence} is equal to
\[ \gamma(\mathcal{B}) = r \gamma(\Omega), \quad \gamma(\Omega) = \max_{(i,j) \in \Omega} \, \la L(\Omega)^{+}(e_i - e_j), e_i - e_j \ra, \] 
where $\mathcal{L}(\Omega)^+$ denotes the pseudoinverse of the Laplacian matrix and $e_1\dots e_N$ are the canonical basis vectors of $\reals^N$. The quantity $\gamma(\Omega) \in [\frac{N}{|\Omega|},1]$ is called the graph \textit{effective resistance} in spectral graph theory \cite{Ellens2011resistance}. $\gamma(\Omega)$ is equal to one if there exists an edge $(i,j)$ such that removing that edge from the graph would make it disconnected. 
On the contrary, $\gamma(\Omega)$ is low if all edges are of equal importance in determining the {connectiveness} of the graph.

According to the previous paragraph, the condition number of $\rho$ w.r.t. the norm induced by $B^{(0)}$ is $\kappa_0 = \sqrt{r |\Omega| \gamma(\Omega)}$, while the choice $B_*$ guarantees $\kappa^* = \sqrt{Nr}$. Therefore, using the optimal preconditioner $B^*$ is beneficial if the graph has \textbf{high effective resistance}, i.e. $|\Omega|\gamma(\Omega) \gg N$.

\section{Numerical experiments}\label{s:numerical}

We consider the following quartic problem
\BEQ\label{eq:toy_pb} \min_{x \in \reals^n} \rho(x) - \la c,x \ra,\quad \rho(x)=\sum_{i=1}^m \la a_i,x \ra^4, 
\EEQ
where $a_1\dots a_m, c \in \reals^n$ are generated randomly. We propose to examine the effect of homogenization, acceleration and preconditioning on different instances.


\begin{figure}[t!]
  \centering
  \subfigure[$\frac{\sigma_{\rm max}(A)}{ \sigma_{\rm min}(A)} = 5$]{
    \includegraphics[width = 0.45\columnwidth]{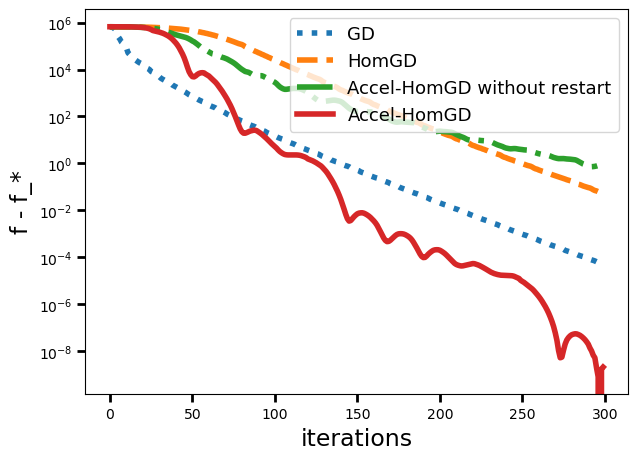}
  }
  \subfigure[$\frac{\sigma_{\rm max}(A)}{ \sigma_{\rm min}(A)} = 50$]{
    \includegraphics[width = 0.45\columnwidth]{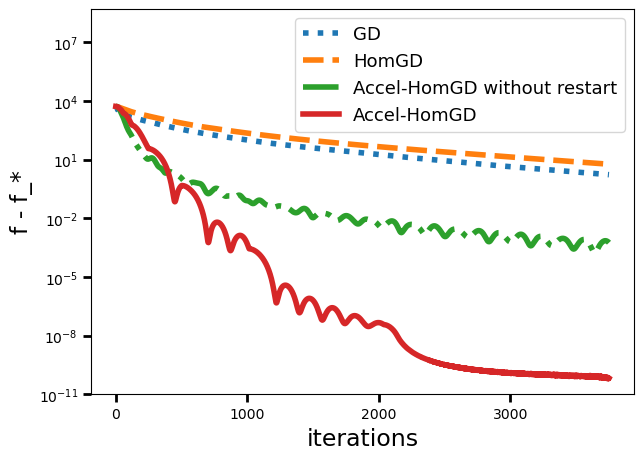}
  }
\caption{\small Comparison of different methods on the synthetic quartic problem \eqref{eq:toy_pb} for different condition numbers of $A$. We plot the objective gap $f(x_k)-f_*$, where $f_*$ is estimated as the minimal objective value computed by all methods. We observe that the restart scheme used in \texttt{Accel-HomGD} (Algorithm \ref{algo:fgd}) is beneficial.}
\label{fig:expes1}

\end{figure}

\subsection{Homogenization and acceleration}

We consider four different algorithms for solving \eqref{eq:toy_pb}. First, we implement our method \texttt{HomGD} (Algorithm \ref{algo:gd}) and its accelerated counterpart \texttt{Accel-HomGD} (Algorithm \ref{algo:fgd}). We also try Algorithm \ref{algo:fgd} \textbf{without restart}, in order to evaluate the practical relevance of the restart technique, and to prove that it is not a mere theoretical construction.
We compare these methods to \texttt{GD}, the standard Euclidean gradient descent method. For all variants, we use the usual Armijo line search strategy to adjust the step size.

We generate instances of \eqref{eq:toy_pb} with random matrices $A \in \reals^{m\times n}$ whose singular values are sampled uniformly in the interval $[\sigma_{\rm min}, 1]$, and random unit Gaussian vectors $c$. We choose $(m,n) = (2000,1000)$. Figure \ref{fig:expes1} reports the results for two different values of~$\sigma_{\rm min}$.

We observe (i) that \texttt{GD} and \texttt{HomGD} have similar practical performance, and therefore that the advantages of \texttt{HomGD} are mostly theoretical (convergence guarantees for fixed step size which does not depend on domain radius); (ii) that accelerated variants improve convergence speed for ill-conditioned problems; (iii) that the restart scheme used in Algorithm \ref{algo:fgd} has a significant impact on convergence speed, and therefore is not just a theoretical construct.

\begin{figure}[t!]
  \centering
  \subfigure[low coherence: $\gamma(A) \approx \frac{n}{m}$]{
    \includegraphics[width = 0.45\columnwidth]{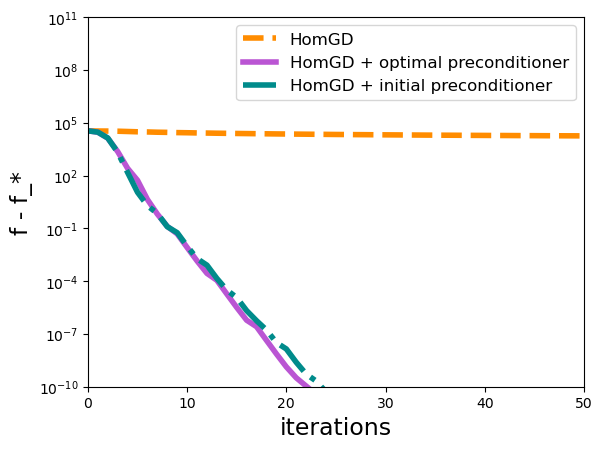}
  }
  \subfigure[high coherence: $\gamma(A) \approx 1$]{
    \includegraphics[width = 0.45\columnwidth]{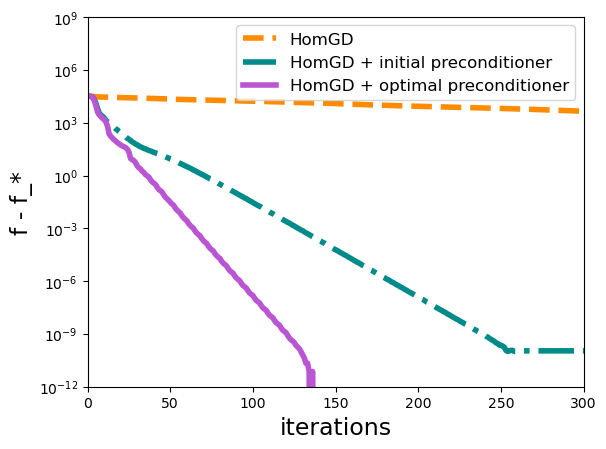}
  }
\caption{\small Performance of the homogenized gradient method (Algorithm \ref{algo:gd}) on Problem \eqref{eq:toy_pb}} with three different preconditioners: \texttt{HomGD}\small uses the standard Euclidean norm ($B = I_n$), \texttt{HomGD + optimal preconditioner}\small uses the norm induced by $B^*$ as computed by Algorithm \ref{algo:fp}, while \texttt{HomGD + initial preconditioner}\small uses the norm induced by $B^{(0)}$, the initial iterate of Algorithm \ref{algo:fp}. As predicted by Section \ref{ss:cond_examples}, the impact of using $B^*$ over $B^{(0)}$ is significant only for highly coherent matrices.

\label{fig:coherence}

\end{figure}

\subsection{Preconditioning} Let us evaluate performance of the preconditioning scheme of Section \ref{s:precond}. Recall that we compared in Section \ref{ss:cond_examples} two preconditioners:
\begin{align}
  &\mbox{Initial preconditioner: } B^{(0)} = \frac{1}{\sqrt{m}}\sum_{i=1}^m a_i a_i^T,  &\kappa_0 = \sqrt{m \gamma(A)} &\in [\sqrt{n},\sqrt{mr}];\\
  &\mbox{Optimal preconditioner: } B^*,  &\kappa_* = \sqrt{n},\hspace{6ex}&
\end{align}
where $\gamma(A)=\max_{i=1\dots m} \la (A^TA)^{-1}a_i,a_i \ra$ is the \textit{coherence} of matrix $A \in \reals^{m\times n}$ with rows $a_1,\dots a_m$. 
Additionnally, since $\sigma_{\rm min}(A)^2 I_n \preceq B^{(0)} \preceq \sigma_{\rm max}(A)^2 I_n$, we can also estimate the condition number w.r.t. the standard Euclidean norm:
\begin{align}
  &\hspace{-5ex}\mbox{No preconditioner } (B = I_n): \hspace{10ex}& \kappa_{I_n} \leq \frac{\sigma_{\rm max}(A)^2}{\sigma_{\rm min}(A)^2} \sqrt{m \gamma(A)}.
\end{align}
We run Algorithm \ref{algo:fp} to compute $B^*$, and compare its performance to $B^{(0)}$ and $I_n$ when used with the homogenized gradient method (Algorithm \ref{algo:gd}).

 We consider problem \eqref{eq:toy_pb} with $m = 500$, $n = 20$. We generate random matrices $A$ which are ill-conditioned, with $\sigma_{\rm max}(A) / \sigma_{\rm min}(A) = 100$, and with different coherence values\footnote{In order to generate matrices $A$ with given coherence, we use the approach of \citep[Section 6.1.]{Ipsen2012}}. Figure \ref{fig:coherence} reports the results for a matrix with high coherence ($\gamma(A) \approx 1$) and a low-coherence matrix ($\gamma(A) \approx \frac{n}{m}$). 

 We observe that (i) since the spectral gap of matrix $A$ is large, using the norm induced by either $B_*$ or $B^{(0)}$ significantly improves performance over the standard Euclidean norm; (ii) the gap between $B_*$ and $B^{(0)}$ is significant only for highly-coherent matrices, as predicted by theory.

\section{Conclusion and perspectives}

We have demonstrated that certain quartic convex problems could be efficiently solved by transforming them into a equivalent \textit{homogenized} problem. The homogenized problem being smooth and convex, it can be solved with standard gradient methods. Our algorithms enjoy {convergence guarantees} with \textbf{fixed step size} and allow for \textbf{acceleration}. Previous methods, such as Euclidean or Bregman gradient descent, were not able satisfy both criteria simultaneously on quartic problems. 

Note that we did not covert the case when the operator $Q$ is positive semidefinite, that is $\alpha = 0$. Under additional technical conditions, we could extend our method to handle this situation. However, we would not benefit from improved rates over those of standard gradient methods for general smooth convex functions.

Naturally, the next goal is to extend this technique to the class of general convex quartic polynomials of the form
\BEQ\label{eq:general_poly} \min_{x \in E} \, Q[x]^4 + P[x]^3 + A[x, x] + \la c, x\ra,
\EEQ
where $Q,P,A$ are multilinear maps of respective degrees 4, 3 and 2, as our current homogenization procedure cannot handle terms of degree 2 and 3 in the objective.

In our second contribution, we studied a method for finding a appropriate preconditioner $B_*$ for the quartic function by using a variant of the Lewis weights. One could ask if this technique can be extended to general polynomials of the form \eqref{eq:general_poly}. Another possible direction for future research is to find a more efficient method for computing $\tau^*$, as the complexity of $\mathcal{\tilde{O}}(mrn^2)$ can be prohibitive for large-scale problems. A way to achieve this would be to compute the quantities $\ell_i(\tau;\mathcal{B})$ only approximatively.

{\small
\paragraph{Acknowledgements.} 
The authors thank the two anonymous reviewers for their constructive comments which greatly contributed to improve the paper quality.

 This paper has received funding from the European Research Council (ERC) under the European Union’s Horizon 2020 research and innovation program (grant agreement No 788368).

  It was also supported by Multidisciplinary Institute in Artificial intelligence MIAI@Grenoble Alpes (ANR-19-P3IA-0003).

}

{\small \bibliographystyle{plainnat_modif}
\bibsep 1ex
\bibliography{library,library_supp}}
\end{document}